\documentclass[reqno,a4paper,11pt]{amsart}

\usepackage[pdfpagelabels]{hyperref}
\usepackage{amssymb}
\usepackage{url}
\usepackage{xcolor}

\newtheorem{theorem}{Theorem}
\newtheorem{lemma}[theorem]{Lemma}
\newtheorem{corollary}[theorem]{Corollary}
\newtheorem{proposition}[theorem]{Proposition}
\newtheorem{oldtheorem}{Theorem}
\newtheorem{oldlemma}[oldtheorem]{Lemma}

\theoremstyle{definition}

\theoremstyle{remark}

\numberwithin{equation}{section}

\newcommand{\Bl}{\Bigl(}
\newcommand{\Br}{\Bigr)}

\newcommand{\diam}{\mathop{\rm diam}}
\newcommand{\supp}{\mathop{\rm supp}}
\newcommand{\ql}{Q^{(l)}}
\def\ve{\varepsilon}

\newcommand{\abs}[1]{\lvert#1\rvert}

\newcommand{\D}{\mathbb{D}}

\newcommand{\N}{\mathbb{N}}

\newcommand{\R}{\mathbb{R}}

\newcommand{\C}{\mathbb{C}}

\newcommand{\e}{\varepsilon}

\newcommand{\KS}{\mathcal{A}^{-\infty}}
\renewcommand{\phi}{\varphi}

\DeclareMathOperator{\dist}{dist}

\newcommand{\A}{\mathcal{A}}

\renewcommand{\a}{\alpha}
\newcommand{\DU}{\overline{D}}
\newcommand{\ph}{D_{\rm\tiny ph}}

\usepackage{graphicx}
\makeatletter
\newcommand\opteq[1]{\mathrel{\mathpalette\opt@eq{#1}}}
\newcommand{\opt@eq}[2]{%
  \begingroup
  \sbox\z@{$#1#2$}%
  \sbox\tw@{\resizebox{!}{.5\ht\z@}{$\m@th#1($}}%
  \nonscript\hskip-\wd\tw@
  \mkern1mu
  \raisebox{-.35\ht\z@}[0pt][0pt]{\resizebox{!}{.5\ht\z@}{$\m@th#1($}}%
  \mkern-1mu
  {#2}%
  \mkern-1mu
  \raisebox{-.35\ht\z@}[0pt][0pt]{\resizebox{!}{.5\ht\z@}{$\m@th#1)$}}%
  \mkern1mu
  \nonscript\hskip-\wd\tw@
  \endgroup
}
\makeatother
\newcommand{\leoq}{\opteq{\leq}}

\begin{document}

\title[Irregular finite order solutions of LDE's]{Irregular finite order solutions of\\ complex LDE's in unit disc}
\thanks{The fifth author was supported in part 
by Ministerio de Ciencia Innovaci\'on y universidades, Spain, projects PGC2018-096166-B-100 and MTM2017-90584-REDT; La Junta de Andaluc\'ia, project FQM210.}

\author{I.~Chyzhykov}
\address{School of Mathematics Science\newline
   \indent        Guizhou  Normal University \newline
      \indent         Guiyang, Guizhou 550001, China}
  
\address{Faculty of Mathematics and Computer Science\newline
\indent Warmia and Mazury University of Olsztyn\newline
\indent S\l{}oneczna 54, Olsztyn, 10710, Poland}

\email{chyzhykov@yahoo.com}

\author{P.~Filevych}
\address{National University Lviv Polytechnic, \newline
\indent 5 Mytropolyt Andrei str., Buiding 4 \newline \indent Lviv 79000, Ukraine}
\email{p.v.filevych@gmail.com}

\author{J.~Gr\"ohn}
\address{Department of Physics and Mathematics\newline
\indent University of Eastern Finland\newline
\indent P.O. Box 111, FI-80101 Joensuu, Finland}
\email{\vspace*{-0.25cm}janne.grohn@uef.fi}

\author{J.~Heittokangas}
\email{\vspace*{-0.25cm} janne.heittokangas@uef.fi}

\author{J.~R\"atty\"a}
\email{jouni.rattya@uef.fi}

\date{\today}



\begin{abstract}
It is shown that the order and the lower order of growth are equal for all non-trivial solutions of $f^{(k)}+A f=0$ if and only if the coefficient $A$ is analytic in the unit disc and $\log^+ M(r,A)/\log(1-r)$ tends to a~finite limit as $r\to 1^-$.
A~family of concrete examples is constructed, where the order of solutions remain the same while the lower order may vary on a~certain interval depending on the irregular growth of the coefficient.
These coefficients emerge as the logarithm of their modulus approximates smooth radial subharmonic functions of prescribed irregular growth on a~sufficiently large subset of the unit disc.  
A~result describing the phenomenon behind these highly non-trivial examples is also established. En route to
results of general nature, a~new sharp logarithmic derivative estimate involving the lower order of growth is discovered.
In addition to these estimates,
arguments used are based, in particular, on the Wiman-Valiron theory adapted for the lower order, and on a~good understanding of the right-derivative of the logarithm of the maximum modulus.
\end{abstract}

\maketitle

\section{Introduction and main results}

The balance between the growth of coefficients and the growth and oscillation of solutions
has been a~central theme of research concerning linear differential equations in a~complex 
domain for over a~half of~century.
In the case of the complex plane~$\C$,
the classical result of Wittich~\cite[Satz~1]{W1966}
states that
the analytic coefficients $A_0, \dotsc, A_{k-1}$ are polynomials 
if and only if all solutions of
	\begin{equation}\label{eq:general}
	f^{(k)} + A_{k-1}f^{(k-1)} + \dotsb + A_1f' + A_0f = 0
	\end{equation}
are entire functions of finite order of growth. Further,
each order belongs to a~finite set of rational numbers
induced by the degrees of the polynomial 
coefficients~\cite{GSW1998}.
In the particular case when the coefficient 
$A\not\equiv 0$ of
	\begin{equation} \label{eq:dek}
	f^{(k)}+Af=0
	\end{equation}
is a polynomial, all non-trivial solutions $f$ of \eqref{eq:dek} have regular growth in the sense that 
the lower order and the order of $f$ are both equal to $\deg(A)/k+1$,~i.e.,
	$$
	\liminf_{r\to\infty}\frac{\log\log M(r,f)}{\log r}= 
	\limsup_{r\to\infty}\frac{\log\log M(r,f)}{\log r}=\frac{\deg(A)}{k}+1;
	$$ 
see~\cite[p.~74]{L:1993} and~\cite[pp.~106--108]{V:1949}. Here $M(r,f)=\max_{|z|=r}|f(z)|$ is the maximum 
modulus of $f$ on the circle $\abs{z}=r$.
Note that, as a polynomial, the coefficient $A$ has regular growth as well in the sense that
	\begin{equation}\label{eq:poly}
	\liminf_{r\to\infty}\frac{\log M(r,A)}{\log r}=\limsup_{r\to\infty}\frac{\log M(r,A)}{\log r}
	=\deg(A)<\infty.
	\end{equation}
The oscillation of solutions is equally intimately connected to the growth of coefficients, see
\cite{CGHR2013, CGHRequiv}.

The connection between the growth of coefficients and the growth of solutions 
is also well understood in the case of the unit disc $\D = \{z\in\C : |z|<1\}$.
In this setting, the analogue of polynomial coefficients in $\C$ are
coefficients belonging to the Korenblum space~$\KS$~\cite{CGHRequiv, K1975}.
Indeed, the analytic coefficients belong to~$\KS$ if and only if all solutions of \eqref{eq:general} are 
of finite order of growth. In the case of \eqref{eq:dek}, all non-trivial solutions are 
of maximal growth and have order of growth uniquely determined by the coefficient 
$A\in\KS$~\cite{CHR2009, CHR2010, CGHRequiv}. 
However, almost nothing is known about the lower order of growth of solutions.
Our principal objective is to show that the degree and the lower degree of growth of the coefficient $A\in\KS$ control 
the lower order of growth
of non-trivial solutions. In particular, we demonstrate that solutions having different order and lower order are
possible in the case of the unit disc, contrasting sharply with  the analogous situation in the complex plane.

To reach concrete statements, some notation is needed.
For a function $f$ analytic in $\D$, the order and the lower order of growth (with respect
to the maximum modulus) are defined by
	\begin{equation*}
	\sigma_M (f) = \limsup_{r\to 1^-} \, \frac{\log^+ \log^+ M(r,f)}{\log{\frac{1}{1-r}}},
	\quad 
	\lambda_M(f) = \liminf_{r\to 1^-} \, \frac{\log^+ \log^+ M(r,f)}{\log{\frac{1}{1-r}}},
	\end{equation*}
respectively. Here the plus sign refers to the non-negative part.
Clearly, $0\leq \lambda_M(f)\leq \sigma_M(f)\leq \infty$ for any analytic $f$ in $\D$.
For a~function $A$ analytic in $\D$, the degree and the lower degree of growth are defined by
	\begin{equation*}
	\sigma_{M,\deg} (A) = \limsup_{r\to 1^-} \, \frac{\log^+ M(r,A)}{\log{\frac{1}{1-r}}},
	\quad 
	\lambda_{M,\deg}(A) = \liminf_{r\to 1^-} \, \frac{\log^+ M(r,A)}{\log{\frac{1}{1-r}}},
	\end{equation*}
respectively. This terminology reflects the polynomial growth in \eqref{eq:poly}. Differing from the
situation in \eqref{eq:poly}, we will construct analytic functions $A$ in $\D$ for which 
$\lambda_{M,\deg}(A)<\sigma_{M,\deg} (A)$, and consider the growth of solutions of \eqref{eq:dek}. 

The 
Korenblum space $\KS$ consists of functions of finite degree, and this implies the following  
statement: All solutions of \eqref{eq:general} are of finite order of growth if and only if
the coefficients are of finite degree. Regarding the non-trivial solutions $f$ of \eqref{eq:dek}, 
\cite[Theorem~1.4]{CHR2010} implies
\begin{equation*}
    \sigma_M(f)
    \leq \max \big\{ 0, \sigma_{M,\deg}(A)/k-1 \big\}
\end{equation*}
and, in particular,
	\begin{equation} \label{eq:growth}
	\sigma_M(f) = 
	\sigma_{M,\deg}(A)/k-1, \quad  \sigma_{M,\deg}(A)\geq 2k.
	\end{equation}

Our first result relates the irregular growth of the coefficient to the irregular growth
of non-trivial solutions of~\eqref{eq:dek}. This shows that the unit disc case is, in a~certain sense, similar to the plane case.


\begin{theorem} \label{thm:new}
Let $k\in\N$ and let $A$ be an~analytic function in $\D$. Then, $\sigma_{M,\deg}(A)=\lambda_{M,\deg}(A) = p >2k$
if and only if $\sigma_M(f) = \lambda_M(f) = p/k-1 >1$ for some (equivalently for all) non-trivial solution(s) $f$ of~\eqref{eq:dek}.
\end{theorem}

The discussion on the principal ingredients of the proof of Theorem~\ref{thm:new} is postponed to the end of the present section.

Theorem~\ref{e:sol_limits2} below unfolds a~family of examples, and shows that 
the difference between $\sigma_{M,\deg}(A)$ and $\lambda_{M,\deg} (A)$ does
not uniquely determine the lower order of solutions of~\eqref{eq:dek}.
The reason why this result differs radically
from the corresponding situation  in the plane
is that the Korenblum space is a much richer family of functions in $\D$ than the set of polynomials is in $\C$.

\begin{theorem}\label{e:sol_limits2}
	Let $k\in\N$, $k\le p_1<p_2<\infty$, $2k<p_2$ and $\alpha\in[p_1/p_2,1]$. 
Then there exists an analytic function $A=A(\alpha)$ in $\D$ such that
$\sigma_{M,\deg} (A)=p_2$, $\lambda_{M,\deg} (A)=p_1$ and
any non-trivial solution $f$ of \eqref{eq:dek} satisfies $\sigma_M(f)=p_2/k-1$ and $\lambda_M(f)=p_1/k-\alpha$.
\end{theorem}

The proof of Theorem~\ref{e:sol_limits2} is rather involved and constitutes the bulk of the paper.
The first part of the proof is a~laborious construction of a~smooth radial subharmonic function $\varphi$ of irregular growth.
Then, we show that there exists an~analytic function $A$ such that $\log |A|$ approximates
$\varphi$ with sufficient precision in a~large subset of $\D$. The upper bound for the lower 
order of growth of any solution of~\eqref{eq:dek}
follows from the irregular growth of the coefficient by a~growth estimate for solutions of linear differential equations.
Meanwhile, the lower bound for the lower order is established by using a~recent integrated logarithmic derivative estimate.
Theorem~\ref{e:sol_limits2} is proved in Section~\ref{sec:proofmain}.

In Theorem~\ref{e:sol_limits2} the order of non-trivial solutions is always $p_2/k-1$, while
the lower order can be any pregiven number on the interval $[p_1/k-1, p_1/k-p_1/p_2]$. In this case, the lower order of growth
is strictly smaller than the order of growth. 
The following theorem reveals 
the general phenomenon 
induced by the irregular growth of the coefficient.


\begin{theorem} \label{th:first}
Let $k\in\N$ and let $A$ be an~analytic function in $\D$ such that
$\sigma_{M,\deg}(A)=p_2\in (0,\infty)$ and $\lambda_{M,\deg}(A)=p_1$.
Then
\begin{enumerate}
\item[\rm (a)]
all solutions $f$ of \eqref{eq:dek} with $\sigma_M(f)>0$ satisfy
	\begin{equation} \label{eq:th12''}
	\frac{p_1}{k} - 1 \leq 1+\left(\lambda_M(f)-\frac{\lambda_M(f)}{\sigma_M(f)}\right)^+;
	\end{equation}
\item[\rm (b)]
if $2k<k\big(2+\tfrac{p_2-2k}{p_2}\big)<p_1\leoq p_2$, then all non-trivial solutions $f$ of \eqref{eq:dek} satisfy
\begin{equation*} 
\lambda_M(f)-\left(1-\frac{\lambda_M(f)}{\sigma_M(f)}\right)
\le\frac{p_1}{k}-1\leoq\frac{p_2}{k}-1=\sigma_M(f).
\end{equation*}
\end{enumerate}
\end{theorem}

The following result shows, analogously to the plane situation in \eqref{eq:poly}, that if $A$ has
regular growth in $\D$, then the solutions of \eqref{eq:dek} have regular growth as well.

\begin{corollary} \label{cor:1} 
Let $k\in\N$ 
and let $A$ be an~analytic function in $\D$ such that
$\sigma_{M,\deg}(A)=p_2$ and $\lambda_{M,\deg}(A)=p_1$, where
$2k<k(2+\frac{p_2-2k}{p_2})<p_1\leq p_2<\infty$.
Then all non-trivial solutions $f$ of \eqref{eq:dek} satisfy
  \begin{equation}\label{e:deviat_est}
  \left|\lambda_M(f)-\frac{p_1}{k}+1\right|\le1-\frac{\lambda_M(f)}{\sigma_M(f)}.
  \end{equation}
\end{corollary}

We may re-write \eqref{e:deviat_est} in the form 
	\begin{equation} \label{eq:corimp}
	\frac{p_1-2k}{p_2-2k} \Bigl( \frac{p_2}{k}-1 \Bigr) \le \lambda_M(f) 
	\le \frac{p_1}{p_2} \Bigl( \frac{p_2}{k}-1 \Bigr).
	\end{equation}
To see that the upper bound in \eqref{eq:corimp} is sharp, choose $\alpha=p_1/p_2$ in 
Theorem~\ref{e:sol_limits2}. It would be desirable to show that the lower bound in~\eqref{eq:corimp}
can be replaced by the value $p_1/k-1$, which corresponds to $\alpha=1$ in Theorem~\ref{e:sol_limits2}. However,
it is not known whether this is true, unless $p_2=p_1$. In this case, \eqref{eq:corimp} reduces to the equality
$\sigma_M(f)=\lambda_M(f)$ by~\eqref{eq:growth}. In fact, we already know this, even under weaker hypothesis, by Theorem~\ref{thm:new}.

The proof of Theorem~\ref{th:first} is given in Section~\ref{sec:th:first}, and it depends on the Wiman-Valiron theory adapted for the lower order of growth. Therefore, we need a~good understanding of the quantities
\begin{equation} \label{eq:def_ls}
	\lambda_*(f)=\liminf_{r\to1^-}\frac{\log^+ K(r,f)}{\log\frac1{1-r}},
        \quad
	\sigma_*(f)=\limsup_{r\to1^-}\frac{\log^+ K(r,f)}{\log\frac1{1-r}},
\end{equation}
where $K(r,f)=r(\log M(r,f))'_+$ for $0\leq r < 1$. Here the plus sign refers to the right derivative.  
In Section~\ref{sec:maxmod}, we prove that each function $f$ analytic in $\D$ satisfies 
$\lambda_*(f)\le\lambda_M(f)+1$ if $f$ is unbounded, 
and $\lambda_M(f)+\frac{\lambda_M(f)}{\sigma_M(f)}\le\lambda_*(f)$ if $\sigma_M(f)>0$. Both estimates are shown to be sharp.

In addition to the Wiman-Valiron theory,
the proof of Theorem~\ref{th:first} strongly relies on a~new logarithmic derivative estimate involving the lower order of growth,
which is stated as Theorem~\ref{TheoremLogDeriv} below.
This result
complements a~known logarithmic derivative estimate for functions of finite maximum modulus order given in terms of a proximate order~\cite{CHR2010}.
To the best of our knowledge, logarithmic derivative estimates involving the lower order of growth do not appear in the existing literature.
The upper density of a measurable set $E\subset[0,1)$ is defined as
    $$
    \DU(E)=\limsup_{r\to1^-}\frac{m_1(E\cap[r,1))}{1-r},
    $$
where $m_1(F)$ denotes the one-dimensional Lebesgue measure of the set $F$.


\begin{theorem}\label{TheoremLogDeriv}
Let $f$ be an analytic function in $\D$ such that $0\leq \lambda_M(f)\leq\sigma_M(f)< \infty$. Let $k$ and $j$ be
integers satisfying $k>j\ge0$, and let $\varepsilon \in (0,1)$. Then there
exist a set $E=E(\varepsilon,f,k,j)\subset[0,1)$ satisfying $\DU(E)=1$ and a constant $C=C(\varepsilon,f,k,j)>0$ such that
    \begin{equation} \label{EqLogEstimate>1} 
    \bigg|\frac{f^{(k)}(z)}{f^{(j)}(z)}\bigg|
    \le C\left(\frac{1}{(1-|z|)^{2+\left(\lambda_M(f)-\frac{\lambda_M(f)}{\sigma_M(f)}\right)^++\varepsilon}}\right)^{k-j}
    \end{equation}
for all $z\in\D$ for which $|z|\in E$.
\end{theorem}

The proof of Theorem~\ref{TheoremLogDeriv} is given in Section~\ref{sec:TLD}, and it is based on~an~approach similar to that in~\cite[Theorem~1.2]{CHR2010}.
In the present paper, we take advantage of the full strength of Linden's result, which is stated as Theorem~\ref{TheoremLinden2} below.
This result provides a~local estimate for the behavior of an~analytic function in terms of a~quantity depending on the maximum modulus,
and it turns out that this quantity can be further estimated in terms of the lower order of growth.

The proof of Theorem~\ref{thm:new} depends on Theorem~\ref{th:first}
and on an~auxiliary result on the lower order of solutions of~\eqref{eq:dek}.
The proof and the auxiliary result are presented in Section~\ref{sec:proofnew}. 
In the final Section~\ref{sec:example} we discuss an~example, which addresses the case when 
the condition $\lambda_{M,\deg}(A)>2k$ in Theorem~\ref{th:first}(b) is not satisfied.
Then  the correlations between 
the growth indicators of the coefficient 
and of solutions become even more complicated.

\section{Proof of Theorem~\ref{e:sol_limits2}} \label{sec:proofmain}

The laborious proof is divided into three parts.

\subsection{Subharmonic functions of irregular growth}

In Lemma~\ref{lemma:app} below, we construct sequences $\{r_n\}$, $\{r_n'\}$,
$\{r_n''\}$, $\{r_n^\star\}$, $\{\widehat{r}_n\}$ each of which satisfy the property
$(1-\rho_{n+1})/(1-\rho_n) \to 0$
as $n\to\infty$. Therefore $1-\rho_n$ decreases to zero faster than any geometric progression, 
inheriting many properties of such progressions.

\begin{lemma} \label{lemma:app}
Let $0<p_1<p_2\le p<\infty$ be constants,
and let $\{\eta_n\}$ be an increasing unbounded sequence such that $\eta_n>1$ for all $n$.

{\rm (a)} There exist a~positive constant $C$ and sequences $\{r_n\}$, $\{r_n'\}$,
$\{r_n''\}$, $\{r_n^\star\}$, $\{\widehat{r}_n\}$, $\{M_n\}$, $\{R_n\}$, $\{\varepsilon_n\}$ such that
$r_n<r_n'\le\widehat{r}_n<r_n^\star<r_n''<r_{n+1}$, $|\varepsilon_n| < (p_2-p_1)/2$, and
\begin{enumerate}
\item[\rm (i)] $\displaystyle \left( \frac{C}{1-r_n'} \right)^{p_1} = \left( \frac{C}{1-r_n} \right)^{p_2+\varepsilon_n}$;

\item[\rm (ii)] $\displaystyle r_n^\star = \widehat{r}_n + \frac{1-\widehat{r}_n}{\log\frac{C}{1-\widehat{r}_n}}$;

\item[\rm (iii)] $\displaystyle R_n = \frac{p_2+\varepsilon_n}{1-r_n} \, r_n$;

\item[\rm (iv)] $\displaystyle \frac{1}{(1-\widehat{r}_n)^{p_2}}=\frac1{(1-r_n')^p};\quad p=p_2\Rightarrow \widehat{r}_n=r_n';$

\item[\rm (v)] $\displaystyle M_n = \frac{p_2-p_1}{(1-\widehat{r}_n)^2} \left( \log\frac{C}{1-\widehat{r}_n} \right)^2$;

\item[\rm (vi)] $\displaystyle
\frac{R_n \log\frac{r_n''}{r_n} + M_n \int_{\widehat{r}_n}^{r_n^\star} \log\frac{r_n''}{t} \, dt-p_1\frac{r_n''-r_n'}{1-r_n'}}{\log\frac{C}{1-r_n''}}
= p_2-p_1 + \varepsilon_{n+1}$;

\item[\rm (vii)] $\displaystyle
\left( R_n + M_n (r_n^\star - \widehat{r}_n) -p_1\frac{r_n''}{1-r_n'}\right)  \frac{1-r_n''}{r_n''} = p_2-p_1 + \varepsilon_{n+1}$;

\item[\rm (viii)] $\displaystyle
r_{n+1} = 1 - \frac{1-r_n''}{\eta_n}$.
\end{enumerate}
Moreover, 
$r_n\to 1^-$ and $\varepsilon_n \to 0$ as $n\to\infty$, and further,
	$$
	1-r_n''\sim\frac{1-\widehat{r}_n}{\log\frac1{1-\widehat{r}_n}},\quad n\to\infty.
	$$

{\rm (b)} If $\varphi:[0,1)\to(0,\infty)$ is given by
\begin{equation*}
\varphi(r) = \begin{cases}
(p_2+\varepsilon_n) \log\frac{C}{1-r}, & r_{n-1}'' \leq r < r_n,\\
(p_2+\varepsilon_n) \log\frac{C}{1-r_n} + R_n \log\frac{r}{r_n}, & r_n \leq r < r_n',\\
p_1 \log\frac{C}{1-r_n'} + R_n\log\frac{r}{r_n} + p_1\left(\log\frac{1-r_n'}{1-r}-\frac{r-r_n'}{1-r_n'}\right), & r_n'\leq r < \widehat{r}_n,\\
p_1 \log\frac{C}{1-r} + R_n\log\frac{r}{r_n} - p_1\frac{r-r_n'}{1-r_n'}+ M_n \int_{\widehat{r}_n}^r \log\frac{r}{t} \, dt, & \widehat{r}_n\leq r < r_n^\star,\\
p_1 \log\frac{C}{1-r} + R_n\log\frac{r}{r_n} - p_1\frac{r-r_n'}{1-r_n'}+ M_n \int_{\widehat{r}_n}^{r_n^\star} \log\frac{r}{t} \, dt, & r_n^\star \leq r < r_n'',\\
\end{cases}
\end{equation*}
where $n\in\N$, then $\varphi$ is continuously differentiable and
\begin{equation} \label{eq:diff}
\frac{1}{r} \, (r \varphi'(r))' =
\begin{cases}
\frac{p_2+\varepsilon_n}{r (1-r)^2}, & r_{n-1}''<r<r_n,\\
0,&r_n<r<r_n',\\
\frac{p_1}{r}\left(\frac1{(1-r)^2}-\frac1{1-r_n'}\right), & r_n'<r<\widehat{r}_n,\\
\frac{p_1}{r}\left(\frac1{(1-r)^2}-\frac1{1-r_n'}\right)+\frac{M_n}{r}, & \widehat{r}_n<r<r_n^\star,\\
\frac{p_1}{r}\left(\frac1{(1-r)^2}-\frac1{1-r_n'}\right), & r_n^\star<r<r_n''.
\end{cases}
\end{equation}
Moreover,
\begin{equation} \label{eq:limits}
\limsup_{r\to 1^-} \frac{\varphi(r)}{\log\frac{1}{1-r}} = p_2, \quad
\liminf_{r\to 1^-} \frac{\varphi(r)}{\log\frac{1}{1-r}} = p_1.
\end{equation}

The radial extension $\varphi: \D \to (0,\infty)$, given by $\varphi(z)=\varphi(|z|)$ for all $z\in\D$, is subharmonic
and satisfies $\Delta\phi(re^{i\theta}) = \frac{1}{r} \, (r \varphi'(r))'$ for all $\theta\in\R$.
\end{lemma}


\begin{proof}
(a) The precise value for the (large) positive constant $C$ depends on many factors,
and will be dealt with later. At this point, it suffices to assume $C\geq e$.
The sequences in the assertion are defined iteratively.
Define $r_0''=0$, $r_1\in (0,1)$ (the precise value will be fixed later) and $\varepsilon_1=0$ for the first iteration.
The general case goes as follows.
If the seed $(r_n,\varepsilon_n)$ is given, where $|\varepsilon_n|<(p_2-p_1)/2$, then
$r_n', \widehat{r}_n, r_n^\star, R_n,M_n$ are
defined by the conditions (i)--(v), which guarantee $r_n<r_n'\le\widehat{r}_n<r_n^\star<1$. We proceed to show that
the system of equations
\begin{equation} \label{eq:system}
\left\{ \begin{aligned}
\frac{R_n \log\frac{r}{r_n} + M_n \int_{\widehat{r}_n}^{r_n^\star} \log\frac{r}{t} \, dt-p_1\frac{r-r_n'}{1-r_n'}}{\log\frac{C}{1-r}}
& = p_2-p_1 + \varepsilon,\\
\left( R_n + M_n (r_n^\star - \widehat{r}_n) -p_1\frac{r}{1-r_n'}\right)  \frac{1-r}{r}& = p_2-p_1 + \varepsilon,
\end{aligned}
\right.
\end{equation}
admits a~unique solution $(r,\varepsilon)$ for $r_n^\star<r<1$. Assuming for a~moment that such solution exists, then
choose $r_n''=r$ and $\varepsilon_{n+1}= \varepsilon$. By defining $r_{n+1}$ via~(viii), we have
obtained the seed $(r_{n+1}, \varepsilon_{n+1})$ for the next iteration.

We still need to prove that~\eqref{eq:system} admits a~unique solution $(r,\varepsilon)$ where $r\in (r_n^\star, 1)$.
By combining the equations, we eliminate $\varepsilon$ and obtain
\begin{align}
& \left( R_n + M_n (r_n^\star - \widehat{r}_n) -p_1\frac{r}{1-r_n'}\right)  \frac{1-r}{r} \log\frac{C}{1-r} \label{eq:combined_1} \\
& \qquad  = R_n \log\frac{r}{r_n} + M_n \int_{\widehat{r}_n}^{r_n^\star} \log\frac{r}{t} \, dt-p_1\frac{r-r_n'}{1-r_n'}. \label{eq:combined_2}
\end{align}
Consider the interval $[r_n^\star,1)$. Let $g_L=g_L(r)$ be the function in~\eqref{eq:combined_1} and let
$g_R=g_R(r)$ be the function in~\eqref{eq:combined_2}. By a straight-forward differentiation,
$$
g_R'(r)=\frac{R_n}{r}+\frac{M_n(r_n^\star-\widehat{r}_n)}{r}-\frac{p_1}{1-r_n'}\ge\frac{1}{r(1-r_n')}\left((p_2-p_1)\log C-p_1\right)>0,
$$
provided $C>e^{\frac{p_1}{p_2-p_1}}$, and hence the function $g_R$ is strictly increasing. Another differentiation gives
\begin{equation*}
\begin{split}
g_L'(r)
&=-\frac{p_1}{1-r_n'}\left(\frac{1-r}{r}\log\frac{C}{1-r}\right)\\
&\quad+\left(R_n+M_n(r_n^\star-\widehat{r}_n)-p_1\frac{r}{1-r_n'}\right)
\left(-\frac{\log\frac{C}{1-r}}{r^2}+\frac1r\right)\\
&<\left(\frac{R_n}{r}+\frac{M_n(r_n^\star-\widehat{r}_n)}{r}-\frac{p_1}{1-r_n'}\right)\left(1-\log C\right)<0
\end{split}
\end{equation*}
whenever $C>e^{\frac{p_2}{p_2-p_1}}$, and hence the function $g_L$ is strictly decreasing.
Therefore there exists at most one point $r\in [r_n^\star,1)$ for which $g_L(r)=g_R(r)$.

Let $a,b$ be positive constants such that $(\log C)^{1/2}> a > b$ (further restrictions apply later), 
and define $\{\alpha_n\}$ and $\{\beta_n\}$ by
\begin{equation*}
  1-\alpha_n
  = a \, \frac{1-\widehat{r}_n}{\left( \log\frac{C}{1-\widehat{r}_n} \right)^{1/2}}, \quad  1-\beta_n
  = b \, \frac{1-\widehat{r}_n}{\left( \log\frac{C}{1-\widehat{r}_n} \right)^2}.
\end{equation*}
Then $r_n^\star < \alpha_n < \beta_n < 1$.  Recall the inequalities $1-x < \log(1/x) < (1/x)(1-x)$,
valid for all $x\in (0,1)$. On one hand, we estimate
	\begin{align*}
	g_L(\alpha_n)
	& > \left(M_n (r_n^\star - \widehat{r}_n)-\frac{p_1}{1-r_n'}\right)\frac{1-\alpha_n}{\alpha_n} \log\frac{C}{1-\alpha_n} \\
	&>a\left(p_2-p_1-p_1\, \frac{1-\widehat{r}_n}{1-r_n'}\, \frac1{\log\frac{C}{1-\widehat{r}_n}}\right)
	\left(\log\frac{C}{1-\widehat{r}_n} \right)^{\frac12}\log\frac{C\left(\log\frac{C}{1-\widehat{r}_n}\right)^\frac12}{a(1-\widehat{r}_n)}\\
	& \geq a\left(p_2-p_1-\frac{p_1}{\log C}\right) \left(\log\frac{C}{1-\widehat{r}_n}\right)^{\frac12}
	\log\frac{C\left(\log\frac{C}{1-\widehat{r}_n}\right)^\frac12}{a(1-\widehat{r}_n)}\\
	& > a\left(p_2-p_1-\frac{p_1}{\log C}\right) \left(\log\frac{C}{1-\widehat{r}_n}\right)^{\frac32},
\end{align*}
and
\begin{align*}
g_R(\alpha_n) & < g_R(1)
 = R_n \log\frac{1}{r_n} + M_n \int_{\widehat{r}_n}^{r_n^\star} \log\frac{1}{t} \, dt-p_1\\
& \le p_2-p_1+\e_n+\frac{p_2-p_1}{r_1}\, \log\frac{C}{1-\widehat{r}_n}.
\end{align*}
Now that we have the factor $\left(\log\frac{C}{1-\widehat{r}_n}\right)^\frac32$ in the lower estimate for $g_L(\alpha_n)$, and the factor $\log\frac{C}{1-\widehat{r}_n}$ in the upper estimate for $g_L(\alpha_n)$, it is easy to see that $g_L(\alpha_n) > g_R(\alpha_n)$ if $a$ (and hence $C$ also) is sufficiently large.
On the other hand, we estimate
\begin{align*}
g_L(\beta_n) \leq \left( p_2 + \e_n + (p_2-p_1) \log\frac{C}{1-\widehat{r}_n} \right)  \frac{b \log\frac{C \left( \log\frac{C}{1-\widehat{r}_n} \right)^{2}}{b(1-\widehat{r}_n)}}{r_1 \left( \log\frac{C}{1-\widehat{r}_n} \right)^{2}},
\end{align*}
and
\begin{align*}
g_R(\beta_n) &> \frac{p_2-p_1}{1-\widehat{r}_n}\log\frac{C}{1-\widehat{r}_n}\left(\beta_n-r_n^\star\right)-p_1\\
&=(p_2-p_1)\log\frac{C}{1-\widehat{r}_n}\left(1-\frac1{\log\frac{C}{1-\widehat{r}_n}}-\frac{b}{\left(\log\frac{C}{1-\widehat{r}_n}\right)^2}\right)-p_1,
\end{align*}
and therefore $g_L(\beta_n) < g_R(\beta_n)$ assuming that $b>0$ is sufficiently small (This reasoning can be easily modified such that in the definitions of $\alpha_n$ and $\beta_n$, the powers $\frac12$ and $2$ are replaced by 1).
As both functions $g_L$ and $g_R$ are continuous, this ensures the existence of
a~unique value $r\in (\alpha_n,\beta_n)$ for which $g_L(r)=g_R(r)$.
This value is denoted by $r_n''$. Define $r_{n+1}$
by (viii) which guarantees $r_n\to 1^-$, as $n\to\infty$, since $\{\eta_n\}$ is unbounded.
Finally, define $\varepsilon_{n+1}$ by~(vi). The reasoning above then proves (vii).

We compute
\begin{align*}
1 \geq \frac{\log\frac{C}{1-\widehat{r}_n}}{\log\frac{C}{1-r_n''}} \geq \frac{\log\frac{C}{1-\widehat{r}_n}}{\log\frac{C}{1-\beta_n}}
=  \left( 1 + \frac{\log\Big( \frac{1}{b} \! \left( \log\frac{C}{1-\widehat{r}_n} \right)^2\Big)}{\log\frac{C}{1-\widehat{r}_n}} \right)^{-1},\quad n\in\N,
\end{align*}
which implies that: (I) for every $\delta>0$ there exists a~constant $C>0$ such that
\begin{equation*}
1 \geq \frac{\log\frac{C}{1-\widehat{r}_n}}{\log\frac{C}{1-r_n''}} \geq \frac{1}{1+\delta},
\end{equation*}
independently of $n$; (II) we also have
\begin{equation*}
\lim_{n\to\infty} \, \frac{\log\frac{C}{1-\widehat{r}_n}}{\log\frac{C}{1-r_n''}}  = 1
\end{equation*}
for any fixed $C$.
By (iii),(v) and (vi),
\begin{equation} \label{eq:alt_vvv}
\begin{split}
 \varepsilon_{n+1}
&  = p_1-p_2 + \frac{p_2+\varepsilon_n}{1-r_n} \, \frac{r_n\,  \log\frac{r_n''}{r_n}}{\log\frac{C}{1-r_n''}} \\
 &  \quad +  \frac{p_2-p_1}{(1-\widehat{r}_n)^2 \log\frac{C}{1-r_n''}} \left( \log\frac{C}{1-\widehat{r}_n} \right)^2 \int_{\widehat{r}_n}^{r_n^\star} \log\frac{r_n''}{t}\, dt
- \frac{p_1\frac{r_n''-r_n'}{1-r_n'}}{\log\frac{C}{1-r_n''}}.
\end{split}
\end{equation}
The right-hand side of~\eqref{eq:alt_vvv} contains four terms. The first term is the constant $p_1-p_2$. Since $|\varepsilon_n| \leq (p_2-p_1)/2$ by the assumption, the second term can be made arbitrarily close
to zero by choosing sufficiently large $C$. By choosing $C$ sufficiently large, the fourth term can be made
arbitrarily close to zero as well. To estimate the third
term, we compute
\begin{align*}
\frac{1}{\widehat{r}_n} & \geq \frac{1}{(1-\widehat{r}_n)^2 \log\frac{C}{1-r_n''}} \left( \log\frac{C}{1-\widehat{r}_n} \right)^2 \int_{\widehat{r}_n}^{r_n^\star} \log\frac{r_n''}{t}\, dt
 \geq \frac{r_n''-r_n^\star}{1-\widehat{r}_n} \, \frac{\log\frac{C}{1-\widehat{r}_n}}{\log\frac{C}{1-r_n''}},
\end{align*}
where
\begin{align*}
1 & \geq \frac{r_n''-r_n^\star}{1-\widehat{r}_n}  = \frac{1-r_n^\star}{1-\widehat{r}_n} - \frac{1- r_n''}{1-\widehat{r}_n}
\geq \frac{(1-\widehat{r}_n)\left( 1 - \left( \log\frac{C}{1-\widehat{r}_n} \right)^{-1} \right)}{1-\widehat{r}_n} - \frac{1- \alpha_n}{1-\widehat{r}_n}\\
 & = 1 - \frac{1}{\log\frac{C}{1-\widehat{r}_n}} - \frac{a}{\left( \log\frac{C}{1-\widehat{r}_n}\right)^{1/2}}.
\end{align*}
We conclude that the third term in the right-hand side of~\eqref{eq:alt_vvv} can be made arbitrarily close to the value $p_2-p_1$
by choosing $r_1\in (0,1)$ sufficiently close to $1$ and $C$ sufficiently large. This means that,
for any $\delta>0$, we obtain $|\varepsilon_{n+1}|< \delta$ by choosing $r_1\in (0,1)$ sufficiently close to $1$ and $C$ sufficiently large.
That said, we may assume $|\varepsilon_{n+1}|<(p_2-p_1)/2$. The same computation also shows that
 $\varepsilon_n\to 0$ as $n\to\infty$. Therefore (vi) ensures the asymptotic equality
$$
1-r_n''\sim\frac{1-\widehat{r}_n}{\log\frac1{1-\widehat{r}_n}},\quad n\to\infty,
$$
which gives a relatively precise location of $r_n''$ with respect to $\widehat{r}_n$.

(b)
Let $\varphi$ be the piecewise-defined function in the assertion. It is immediate that
$\varphi$ is continuous in each subinterval. At endpoints of subintervals, we conclude that
continuity at $r_n$ is clear; continuity at $r_n'$ follows from (i); continuity at $\widehat{r}_n$ is clear and the same happens with $r_n^\star$; and
continuity at $r_n''$ follows from (vi). By straight-forward differentiation,
\begin{equation*}
\varphi'(r) = \begin{cases}
\frac{p_2+\varepsilon_n}{1-r}, & r_{n-1}'' \leq r < r_n,\\
\frac{R_n}{r}, & r_n \leq r < r_n',\\
\frac{R_n}{r}+p_1\left(\frac1{1-r}-\frac1{1-r_n'}\right), & r_n' \leq r < \widehat{r}_n,\\
\frac{R_n}{r}+p_1\left(\frac1{1-r}-\frac1{1-r_n'}\right) + M_n \frac{r-\widehat{r}_n}{r}, & \widehat{r}_n\leq r < r_n^\star,\\
\frac{R_n}{r} +p_1\left(\frac1{1-r}-\frac1{1-r_n'}\right)+ M_n \frac{r_n^\star-\widehat{r}_n}{r} & r_n^\star \leq r < r_n''.\\
\end{cases}
\end{equation*}
Also the derivative $\varphi'$ is continuous in each subinterval. At endpoints of subintervals, we conclude that
continuity at $r_n$ follows by (iii); continuity at $r_n'$, $\widehat{r}_n$ and $r_n^\star$ is clear; and
continuity at $r_n''$ follows from (vii). Another straight-forward differentiation gives \eqref{eq:diff}.

We need to prove the properties in \eqref{eq:limits}. Note that
\begin{equation*}
\lim_{n\to\infty} \, \frac{\varphi(r_n)}{\log\frac{1}{1-r_n}} = p_2, \quad
\lim_{n\to\infty} \, \frac{\varphi(r_n')}{\log\frac{1}{1-r_n'}}
= \lim_{n\to\infty} \, \frac{p_1\log\frac{C}{1-r_n'} + \frac{p_2+\varepsilon_n}{1-r_n} \, r_n \log\frac{r_n'}{r_n}}{\log\frac{1}{1-r_n'}}  = p_1.
\end{equation*}
We only give details on cases where the upper and lower estimates are different.
The remaining computations are similar and hence omitted.
{\color{black}
For $r_{n-1}''\leq r <r_n$, we deduce the upper estimate
\begin{equation*}
\frac{\varphi(r)}{\log\frac{1}{1-r}} \leq (p_2+\varepsilon_n) \, \left( 1 + \frac{\log C}{\log\frac{1}{1-r_{n-1}''}} \right) \longrightarrow p_2, \quad n\to\infty,
\end{equation*}
and the lower estimate
\begin{equation*}
\frac{\varphi(r)}{\log\frac{1}{1-r}} \geq (p_2+\varepsilon_n) \, \left( 1 + \frac{\log C}{\log\frac{1}{1-r_n}} \right) \longrightarrow p_2, \quad n\to\infty.
\end{equation*}}

For $r_n \leq r < r_n'$, we have the upper estimate
\begin{align*}
\frac{\varphi(r)}{\log\frac{1}{1-r}}
& = (p_2+\varepsilon_n) \, \frac{\log\frac{C}{1-r_n}}{\log\frac{1}{1-r}} + \frac{p_2+\varepsilon_n}{1-r_n} \, \frac{r_n \log\frac{r}{r_n}}{\log\frac{1}{1-r}}\\
& \leq  (p_2+\varepsilon_n) \, \left( 1 + \frac{\log C}{\log\frac{1}{1-r_n}} \right) + \frac{p_2+\varepsilon_n}{\log\frac{1}{1-r_n}} \longrightarrow p_2, \quad n\to\infty,
\end{align*}
and, by (i), the lower estimate
\begin{align*}
\frac{\varphi(r)}{\log\frac{1}{1-r}}
& \geq  (p_2+\varepsilon_n) \, \frac{\log\frac{C}{1-r_n}}{\log\frac{1}{1-r_n'}}
=   \frac{p_1 (p_2+\varepsilon_n) \log\frac{C}{1-r_n} }{(p_2+\varepsilon_n) \log\frac{C}{1-r_n}- p_1\log C} \longrightarrow p_1, \quad n\to\infty.
\end{align*}

{\color{black}
For $r_n'\leq r <\widehat{r}_n$, we have the upper estimate
\begin{align*}
\frac{\varphi(r)}{\log\frac{1}{1-r}}
& = \frac{p_1 \log \frac{C}{1-r_n'}}{\log\frac{1}{1-r}} + \frac{p_2+\varepsilon_n}{1-r_n} \, \frac{r_n \log\frac{r}{r_n}}{\log\frac{1}{1-r}}
+p_1 \, \frac{\log\frac{1}{1-r} - \log\frac{1}{1-r_n'}- \frac{r-r_n'}{1-r_n'}}{\log\frac1{1-r}}\\
& = p_1 + \frac{\frac{p_2+\varepsilon_n}{1-r_n}\, r_n \, \log\frac{r}{r_n} - p_1 \, \frac{r-r_n'}{1-r_n'} + p_1 \log C}{\log\frac{1}{1-r}} \\
& \leq  p_1 + \frac{p_2+\varepsilon_n}{\log\frac{1}{1-r_n'}} + \frac{p_1\log C}{\log\frac{1}{1-r_n'}}
\longrightarrow p_1, \quad n\to\infty.
\end{align*}
The following lower estimate is immediate
\begin{equation*}
\frac{\varphi(r)}{\log\frac{1}{1-r}}
 \geq p_1\left(1+\frac{\log C}{\log\frac{1}{1-\widehat{r}_n}}\right) - p_1 \frac{\frac{\widehat{r}_n-r_n'}{1-r_n'}}{\log\frac1{1-r_n'}}
\longrightarrow p_1, \quad n\to\infty.
\end{equation*}

For $\widehat{r}_n\leq r <r_n^\star$, we have the upper estimate
\begin{align*}
\frac{\varphi(r)}{\log\frac{1}{1-r}}
& = p_1 \, \frac{\log\frac{C}{1-r}}{\log\frac{1}{1-r}} + \frac{p_2+\varepsilon_n}{1-r_n} \, \frac{r_n \log\frac{r}{r_n}}{\log\frac{1}{1-r}}
-\frac{p_1\frac{r-r_n'}{1-r_n'}}{\log\frac1{1-r}}\\
& \qquad + \frac{p_2-p_1}{(1-\widehat{r}_n)^2} \left( \log\frac{C}{1-\widehat{r}_n} \right)^2 \, \frac{\int_{\widehat{r}_n}^r \log\frac{r}{t}\, dt}{\log\frac{1}{1-r}}\\
& \leq  p_1  \left(1+\frac{\log C}{\log\frac{1}{1-\widehat{r}_n}}\right)+ \frac{p_2+\varepsilon_n}{\log\frac{1}{1-\widehat{r}_n}}
+ \frac{p_2-p_1}{\widehat{r}_n} \frac{1}{\log\frac{1}{1-\widehat{r}_n}} \longrightarrow  p_1, \quad n\to\infty,
\end{align*}
and the lower estimate
\begin{equation*}
\frac{\varphi(r)}{\log\frac{1}{1-r}}
 \geq  p_1  \left(1+\frac{\log C}{\log\frac{1}{1-r_n^\star}}\right)
-\frac{p_1}{\log\frac1{1-\widehat{r}_n}} \longrightarrow p_1, \quad n\to\infty.
\end{equation*}}

For $r_n^\star\leq r <r_n''$, we have, by (ii), the upper estimate
\begin{align*}
\frac{\varphi(r)}{\log\frac{1}{1-r}}
& = p_1 \, \frac{\log\frac{C}{1-r}}{\log\frac{1}{1-r}} + \frac{p_2+\varepsilon_n}{1-r_n} \, \frac{r_n \log\frac{r}{r_n}}{\log\frac{1}{1-r}}
-\frac{p_1\frac{r-r_n'}{1-r_n'}}{\log\frac1{1-r}}\\
& \qquad + \frac{p_2-p_1}{(1-\widehat{r}_n)^2} \left( \log\frac{C}{1-\widehat{r}_n} \right)^2 \, \frac{\int_{\widehat{r}_n}^{r_n^\star} \log\frac{r}{t}\, dt}{\log\frac{1}{1-r}}\\
&\leq p_1    \left(1+\frac{\log C}{\log\frac{1}{1-r_n^\star}}\right) + \frac{p_2+\varepsilon_n}{\log\frac{1}{1-r_n^\star}}
+ \frac{p_2-p_1}{\widehat{r}_n} \, \frac{\log\frac{1}{1-\widehat{r}_n}}{\log\frac{1}{1-r_n^\star}}
\longrightarrow  p_2, \quad n\to\infty,
\end{align*}
and the lower estimate
\begin{equation*}
\frac{\varphi(r)}{\log\frac{1}{1-r}}
  \geq p_1 -\frac{p_1}{\log\frac1{1-r_n^\star}}
\longrightarrow p_1, \quad n\to\infty.
\end{equation*}
This completes the proof of (b).
\end{proof}

\subsection{Approximation of subharmonic functions}

We will find an~analytic function $A$ in $\D$ such that $\log|A|$ approximates the radial function 
$\varphi$, constructed in Lemma~\ref{lemma:app}, with sufficient precision (logarithmic growth is approximated
at a $\log\log$ accuracy).
In order to do this we use \cite[Theorem~1]{ChyLyu}, stated as Theorem~\ref{t:3} below.
Some more notation is needed. We say that a~differentiable function $b:[0,1)\to(0,1)$ is of regular variation 
if $r+b(r)<1$ for all $0\le r<1$,
	\begin{equation*} 
	b(r+b(r))\asymp b(r),\quad 0\le r<1,
	\end{equation*}
and 
	\begin{equation*} 
	\sup_{0<r<1}\int_{cb(r)}^{\frac 34(1-r)} \frac{\sup \{ |b'(\rho)|: \rho\in[0,1), \, |\rho-r|\le \tau \}}{\tau}\,d\tau<\infty
	\end{equation*}
for each $0<c<3/4$.
Since $b(r)<1-r$ by the definition, $b(r)\to 0^+$ as $r\to 1^-$ and $cb(r)< 3(1-r)/4$ for all $0<c<3/4$. 
Observe that $b$ is not required to be monotonic.
Standard computations show that the
decreasing functions 
$b_1(r)=(1-r)/C$ and $b_2(r)=(1-r)/(\log(C/(1-r)))$ for $C>1$ are of regular variation. These functions play
a~role in our application.

A polar rectangle is a set of the form $Q=\{z=re^{i\theta}:\,r_1<r<r_2,\, \theta_1<\theta<\theta_2\}\subset\D$. 
We write
$\ell_r(Q)=r_2-r_1$ and $\ell_\theta(Q)=\theta_2-\theta_1$.
We say that a~positive Borel measure  $\mu$ in $\mathbb{D}$ admits a regular partition, with respect to a~function $b$ of regular variation, if there exist a sequence $\{\ql\}$ of polar rectangles in $\D$ and a~decomposition $\mu=\sum\mu^{(l)}$ such that
\begin{itemize}
\item[(i)] $\supp \mu^{(l)} \subset \ql$ and $\mu^{(l)}(\ql)=2$ for all $l$;
\item[(ii)] there exists $N\in\N$ such that each $z\in\D$ belongs to at most $N$ different rectangles $\ql$;
\item[(iii)] there exists a~constant $c>0$ such that 
	$c^{-1} < \ell_\theta(\ql) /\ell_r(\ql) <c$ for all~$l$;
\item[(iv)] $\diam \ql \asymp b(\dist (\ql, 0))$ for all~$l$.
\end{itemize}
Further, $\mu$ is locally regular with respect to a function $b$ of regular variation if there exists $r_0\in(0,1)$ such that
	\begin{equation} \label{eq:stt}
	\sup_{r_0<|z|<1} \, \int_0^{b(|z|)}\frac{\mu(D(z,t))}t \, dt <\infty.
	\end{equation}
Note that (i) the behavior of $\mu$ in compact subsets of $\D$ is not relevant, since $b(r)\to 0^+$ as $r\to 1^-$; (ii) $\mu$ is not necessarily finite, as the hyperbolic measure $d\mu(z)=dm_2(z)/(1-|z|^2)^2$ is locally
regular with respect to $b_1$ (here and later $m_2$ is the two-dimensional Lebesgue measure); (iii) by Fubini's theorem, the condition~\eqref{eq:stt} is equivalent to
\begin{equation*}
\sup_{r_0<|z|<1} \, \int_{D(z,b(|z|))} \log\frac{b(|z|)}{|\zeta-z|}\, d\mu(\zeta) < \infty.
\end{equation*}
Finally, let $Z_f$ stand for the zero set of an analytic function $f$.

Recall that each subharmonic function $u: \D \to \R$ induces a~unique positive measure, namely the Riesz measure 
$\mu = \mu_u$, 
such that $d\mu =\frac1{2\pi} \Delta u\, dm_2$.


\begin{oldtheorem}[\protect{\cite[Theorem~1]{ChyLyu}}] \label{t:3}
Let $u: \D \to \R$ be a subharmonic function
such that $\mu_u$ admits  a~regular partition and $\mu_u$ is locally regular,
with respect to some $b:[0,1)\to(0,1)$ of regular variation.
Then there
exists an analytic function $f$ in~$\D$ such that
	\begin{equation*} 
	\sup_{z\in\D}\big( \log|f(z)| -u(z) \big)<\infty,
	\end{equation*}
and for each $\varepsilon>0$ there exist $r_1 = r_1(\varepsilon)\in(0,1)$ and $C=C(\varepsilon)>0$ such that
	\begin{equation*} 
	\big| \log|f(z)|-u(z) \big| \leq C, \quad r_1<|z|<1, \quad z\not\in E_\ve,
	\end{equation*}
where $E_{\ve}=\{ z\in \D: \dist (z, Z_f)\le \ve b(|z|)\}$. Moreover, the zero-set of $f$ satisfies
$Z_f\subset \bigcup_{\zeta\in \supp \mu_u} D(\zeta, K \, b(|\zeta|))$ for some $K>0$.
\end{oldtheorem}

\begin{lemma}\label{lemma:approx}
Let $\varphi$ be as in~Lemma~\ref{lemma:app}.
Then there exists an analytic function $A$ in~$\D$ with the following properties:
\begin{enumerate}

\item[\rm (a)] for each $\varepsilon\in(0, 1/2)$ there exist $\rho_1=\rho_1(\varepsilon)\in(0,1)$ and a constant 
$C_2 = C_2(\varepsilon)>0$ such that
	$$
	\big|\log|A(z)|-\varphi(z)\big|<C_2+ C_2 \log \log \frac{1}{1-|z|} , \quad z\not \in E_\varepsilon,\quad \rho_1<|z|<1,
	$$
where $E_\varepsilon=\{z\in\D:\dist(z, Z_A)\le\varepsilon (1-|z|)\}$;

\item[\rm (b)] there exists a~constant $C_3>0$ such that 
$$ \big| \log M(r,A)-\varphi(r) \big| \leq C_3\log \log \frac{1}{1-r}, \quad \rho_1<r<1.$$
\end{enumerate}
Moreover, $Z_A\subset\bigcup_{\zeta\in \supp \mu} D(\zeta, (1-|\zeta|)/2)$ and
there exists a~constant $C_4>0$ such that
\begin{equation*}
m_1\big( E_\varepsilon \cap \partial D(0,r) \big) \leq C_4 \varepsilon, \quad r\in (1/2,1).
\end{equation*}
\end{lemma}


\begin{proof} 
Recall that $r_0''=0$ and  define $r_{n,0}=r''_{n-1}$ for all $n\in\N$.
Define the annuli as in~\eqref{eq:diff}:
\begin{align*}
A_n & = \big\{ z\in\D : r_{n-1}''< |z| < r_n \big\},\\
A_n' & =\big\{ z\in\D : r_n< |z| < r_n' \big\}, \\
\widehat{A}_n & = \big\{ z\in\D : r_n'< |z| < \widehat{r}_n \big\},\\
A_n^\star & = \big\{ z\in\D : \widehat{r}_n< |z| < r_n^\star \big\},\\
A_n'' & = \big\{ z\in\D : r_n^\star< |z| < r_n'' \big\},
\end{align*}
and note that the restriction of the Riesz measure $\frac 1{2\pi} \Delta \varphi $ of the function $\varphi$ on the boundary of any annulus is the zero measure.
Define inductively $r_{n,k}$, $k\in \{0, \dots, m_n\}$, where $m_n$ will be specified later.  
Given $r_{n,k}$ we choose the  numbers $0=\psi_{n,k}^0< \dots < \psi_{n,k}^{[\frac1{1-r_{n,k}}]}=2\pi$ 
such that $\psi_{n,k}^j-\psi_{n,k}^{j-1}=\frac{2\pi}{[\frac1{1-r_{n,k}}]}$,
where the square brackets denote
the integer part of a~real number.
In order to divide the annulus $ \{ \zeta: r_{n,k}<|\zeta|<1\}$ into the polar rectangles, define
\begin{equation*}
Q_{n,k}^j= \big\{ \zeta: r_{n,k}<|\zeta|<r_{n,k+1}, \, \, \psi_{n,k}^{j-1}\le \arg \zeta <\psi_{n,k}^{j} \big\},
\quad j=1,\dotsc, \left[ \frac{1}{1-r_{n,k}} \right],
\end{equation*}
 where  $r_{n,k+1}$ is chosen  such that $\mu_\varphi(Q_{n,k}^j)=2$ for all $j$, that is,
\begin{equation*}
\begin{split}
2&=\frac 1{2\pi} \int_{\psi_{n,k}^{j-1}}^{\psi_{n,k}^{j}} \int_{r_{n,k}}^{r_{n,k+1}} \frac{p_2+\varepsilon_{n-1}}{r(1-r)^2}
  \, r dr d\theta
 =\frac{p_2+\varepsilon_{n-1}}{[\frac{1}{1-r_{n,k}}]} \int_{r_{n,k}}^{r_{n,k+1}} \frac{dr}{(1-r)^2}\\
&=\frac{(r_{n,k+1}-r_{n,k})(p_2+\varepsilon_{n-1})}{[\frac{1}{1-r_{n,k}}](1-r_{n,k})(1-r_{n,k+1})},
\quad j=1,\dotsc, \left[ \frac{1}{1-r_{n,k}} \right].
\end{split}
\end{equation*}
Such $r_{n,k+1}$ exists and is unique, because $(x-r_{n,k})/(1-x)$ increases to infinity as $x\to 1^-$.
In particular,
$$ r_{n,k+1}-r_{n,k}\sim \frac{2}{p_2}(1-r_{n, k+1}), \quad n\to\infty.$$

We continue the process until 
the interval $ (r_{n,m}, r_{n,m+1}]$ contains $r_{n}$, that is, when either $r_{n,m+1}=r_n$ or $r_{n,m+1}>r_n'$, because
the Laplacian of $\varphi$ vanishes on~$A_n'$. Set $m_n=m$.
We redefine $Q_{n,m_n-1}^j$ as
 $$Q_{n,m_n-1}^j= \big\{ \zeta: r_{n,m_n-1}<|\zeta|<r_{n}, \, \, \widetilde \psi_{n,m_n-1}^{j-1}\le \arg \zeta <\widetilde \psi_{n,m_n-1}^{j} \big\},$$
where $\widetilde\psi_{n,m_n-1}$ is defined in the following way. We set $\widetilde\psi_{n, m_n-1}^0=0$
and define $\widetilde\psi_{n, m_n-1}^j$ for $1\le j \le s_n-1$ by the equation
$\mu_\varphi(Q_{n,m_n-1}^j)=2$, where the value $s_n$ is uniquely defined by 
the condition $s_n=\min \{ j: \widetilde\psi_{n, m_n-1}^j > 2\pi\}$.
Redefine $Q_{n,m_n-1}^{s_n-1}$ so that the collection 
\begin{equation*}
\left( \bigcup_{k=0}^{m_n-2}\bigcup_{j=1}^{[ 1/(1-r_{n,k}) ]} \overline{Q_{n,k}^j} \right) \cup \bigcup_{j=1}^{s_n-1} \overline{Q_{n,m_n-1}^j} 
\end{equation*}
covers the whole annulus $A_n$. Let $Q_{n1}$ denote this modified rectangle $Q_{n,m_n-1}^{s_n-1}$ and note that 
its side lengths are comparable to $1-r_n$ while $2\leq \mu_\varphi(Q_{n1}) < 4$. 

Since on $\widehat A_n$ and $A_n''$ we have
	$$
	\Delta\varphi(re^{i\theta})
	=\frac{p_1}{r} \left(\frac{1}{(1-r)^2}-\frac 1{1-r_n'}\right)
	\sim\frac{p_1}{(1-r)^2}, \quad  n\to \infty,
	$$
 we can similarly define partitions 
\begin{align*}
  \{ \widehat Q_{n,k}^j \}, & \! & 0\le k\le \widehat{m}_n-1, &  & 1\le j\le \widehat{s}_n-1,\\
 \{  {Q''}_{n,k}^j \}, & & 0\le k\le {m}_n''-1, & & 1\le j\le {s}_n''-1,
\end{align*}
of the Riesz measures of the restrictions $\varphi\bigr|_{\widehat A_n}$ and $\varphi\bigr|_{A''_n}$, respectively. Then
\begin{enumerate}
\item[\rm (i)]
$\mu_\varphi(\widehat Q_{n,k}^j)=2$
for all $0\le k\le \widehat{m}_n-2$ and $1\le j\le [ 1/(1-\widehat r_{n,k}) ]$, and $\mu_\varphi(\widehat Q_{n,\widehat m-1}^j)=2$
for all $1\le j\le \widehat{s}_n-2$;
\item[\rm (ii)]
$\widehat Q_{n,\widehat{m}_n-1}^{\widehat{s}_n-1}$ is modified, and denoted by $Q_{n2}$, 
such that 
\begin{equation*}
\left( \bigcup_{k=0}^{\widehat{m}_n-2}\bigcup_{j=1}^{[ 1/(1-\widehat{r}_{n,k}) ]} \overline{\widehat Q_{n,k}^j} \right) \cup \bigcup_{j=1}^{\widehat{s}_n-1} \overline{\widehat Q_{n,\widehat m_n-1}^j} 
\end{equation*}
covers the whole annulus $\widehat A_n$, which implies $2 \leq \mu_\varphi(Q_{n2}) < 4$;
\item[\rm (iii)]
$\mu_\varphi({Q''}_{n,k}^j)=2$
for all $0\le k\le m''_n-2$ and $1\le j\le [ 1/(1-r''_{n,k}) ]$, and $\mu_\varphi({Q''}_{n,m''_n-1}^j)=2$
for all $1\le j\le s''_n-2$;
\item[\rm (iv)]
${Q''}_{n,m''_n-1}^{s''_n-1}$ is modified, and denoted by $Q_{n3}$, 
such that 
\begin{equation*}
\left( \bigcup_{k=0}^{m''_n-2}\bigcup_{j=1}^{[ 1/(1-r''_{n,k}) ]} \overline{{Q''}_{n,k}^j} \right) \cup \bigcup_{j=1}^{s''_n-1} 
\overline{{Q''}_{n,m''_n-1}^j} 
\end{equation*}
covers the whole annulus $A''_n$, which implies $2 \leq \mu_\varphi(Q_{n3}) < 4$.
\end{enumerate}
By construction, $ Q_{n2}$ and ${Q}_{n3}$ are polar rectangles with side lengths comparable to $1-\widehat r_n$ and $1-r_n''$, respectively. 

Then consider the annulus  $ A_n^*$. We choose
$\psi _n^{*0}=0$, and $\psi_{n}^{*j}$ such that  for 
$Q_n^{*j}= \{ \zeta: \widehat r_{n}<|\zeta|<r_n^*, \widetilde\psi_{n}^{j-1}\le \arg \zeta <\widetilde\psi_{n}^{j} \}$, we have
\begin{equation*}
\begin{split}
 2&=\frac1{2\pi} \int_{Q_n^{*j}} \Delta \varphi(\zeta) \, dm_2(\zeta)
	=\frac{\psi_{n}^{*j}-\psi_{n}^{*j-1}}{2\pi} \int_{\widehat r_n}^{r_n^*} \Delta \varphi(r) r\, dr\\
	&=(1+o(1)) \frac{\psi_{n}^{*j}-\psi_{n}^{*j-1}}{2\pi} \int_{\widehat r_n}^{r_n^*} \frac{p_2-p_1}{(1-\widehat r_n)^2}
	\log^2\frac C{1-\widehat r_n}\,dr\\
	&=(1+o(1))\frac{\psi_{n}^{*j}-\psi_{n}^{*j-1}}{2\pi}\frac{(p_2-p_1)\log\frac C{1-\widehat r_n}}{1- \widehat r_n} , \quad n\to \infty.
\end{split}
\end{equation*}
The value $s_n^*$ is uniquely defined by the condition $s_n^*=\min \{ j: {\psi^*}_{n, m-1}^j > 2\pi\}$.
Then  we obtain  a partition $\{  Q_{n}^{*j} \}$,  $1\le j\le {s}^*_n-1$,  of the Riesz measure of $\varphi\bigr|_{A_n^*}$ such that
\begin{enumerate}
\item[\rm (i)]
$\mu_\varphi( Q_{n}^{*j})=2$ for all $1\leq j \leq {s}^*_n-2$;
\item[\rm (ii)]
$Q_{n}^{*({s}^*_n-1)}$ is modified, and denoted by $Q_{n4}$, such that
$\bigcup_{j=0}^{s_n^*-1} \overline{ Q_{n}^{*j}}$  
covers the whole annulus $A_n^*$, which implies $2\leq \mu_\varphi(Q_{n4})<4$.
\end{enumerate}
By construction, $ Q_{n4}$ is a~polar rectangle with side lengths comparable to $\frac{1-\widehat r_n}{\ln \frac1{1-\widehat r_n}}$.

 Let $d\nu=\Delta \varphi \, dm_2\bigr|_{\bigcup_{n=1}^\infty (A_n^*\setminus Q_{n4})}$. We define
 \begin{align}%
   u_2(z) & =\int _{\D} \Bl \log \left| \frac{\overline{\zeta}(\zeta-z)}{1-\overline\zeta z}\right|+{\rm Re}\, \frac{1-|\zeta|^2}{1-\overline\zeta z} \Br \, d\nu(\zeta), \label{eq:under}\\
   u_1(z)& =\varphi(z) -u_2(z) -\sum_{j=1}^4 u_j^*(z), \notag
 \end{align}
 where
 \begin{equation}\label{e:u_star_def}
 u_j^*(z)=\sum_{n=1}^\infty \int_{Q_{nj}} \log \Bigl| \frac{z-\zeta}{1-z\overline\zeta}\Bigr| \Delta \varphi(\zeta)\, dm_2(\zeta).
 \end{equation}
The fact that these integrals converge depends on $\{ \eta_n\}$ in Lemma~\ref{lemma:app}.
Careful inspection shows that the condition $\eta_n\to\infty$, as $n\to\infty$, is sufficient for our purposes.

Although the subharmonicity of $u_2$ follows by standard arguments \cite{Chy_israel,HK},
we discuss it in detail for the convenience of the reader. More precisely, we show that $u_2$ is well-defined,
it is subharmonic in $\D$ and its Riesz measure coincides with $\nu/(2\pi)$.

First observe that the integrand in \eqref{eq:under}, as the logarithm of modulus of the primary factor of genus 1, admits the estimate  
\begin{align*}
\left|\log \left| \frac{\overline{\zeta}(\zeta-z)}{1-\overline\zeta z}\right|+{\rm Re}\, \frac{1-|\zeta|^2}{1-\overline\zeta z} \right| \le 2  \frac{(1-|\zeta|^2)^2}{|1-\overline\zeta z|^2}, \quad \frac{1-|\zeta|^2}{|1-\overline \zeta z|}<\frac 12,
\end{align*}
by \cite[Chap.V.10, p.223]{Tsuji}. We have $$
\nu(A_n^*) \asymp \int_{\widehat{r}_n}^{r_n^*} \frac{1}{(1-\widehat{r}_n)^2}  \log ^2  \frac  {C}{1-\widehat{r}_n} \, dr\asymp  \frac{1}{1-\widehat{r}_n}  \log   \frac  {C}{1-\widehat{r}_n}, \quad n\to \infty.$$
By Lemma~\ref{lemma:app}(viii) we get $1-r_{n+1} = (1-r_n'')/\eta_n \leq (1-r_n)/\eta_n$,
and hence
$1- \widehat r_n \geq 1 - r_{n+1} \geq \eta_{n+1}(1-r_{n+2}) \geq \eta_{n+1}(1-\widehat r_{n+2}) \geq \eta_1 (1-\widehat r_{n+2})$. 
We obtain
\begin{align*}
|u_2(z)| \lesssim \int_{\D} (1-|\zeta|)^2 \, d\nu(\zeta)
\lesssim \sum_{n=1}^\infty (1-\widehat r_n) \log \frac{C}{1-\widehat{r}_n} < \infty
\end{align*}
uniformly in  $z\in F$, where $F$ is a compact subset of $\D$. 

To prove that $u_2$ is subharmonic in $\D$, it is sufficient to verify it locally.
First, we need an upper estimate of $\nu$ on compact subsets of $\D$.
For any $r_1^*\leq r<1$, there exists $N\in\N\setminus\{1\}$ such that $r_{N-1}^*\leq r < r_N^*$.
Since $(1-\widehat {r}_{n+1})/(1-\widehat {r}_n)\to 0$ as $n\to \infty$, we deduce 
\begin{align*}
\nu\big( D(0,r) \big) & =\sum_{n=1}^{N-1} \nu(A_n^*)+ \nu(D(0,r)\cap A_N^*)\\
& \lesssim  \sum_{n=1}^{N-1}  \frac{1}{1-\widehat{r}_n}  \log   \frac  {C}{1-\widehat{r}_n} +\nu(D(0,r)\cap A_N^*)\\
& \lesssim     \frac{1}{1-\widehat{r}_{N-1}}   \log   \frac  {C}{1-\widehat{r}_{N-1}} +  \frac{1}{1-r}   \log   \frac  {C}{1-{r}}\\
& \lesssim      \frac{1}{1-r}   \log   \frac  {C}{1-{r}}.
\end{align*}
Second, let $K$ be an arbitrary~compact subset of $\D$. Note that
\begin{align*}
u_2(z) = \int _{K} \log|\zeta - z | \, \nu(\zeta)+ \int_K h(z,\zeta) \, \nu(\zeta)
+ \int _{\D\setminus K} \big( \log|\zeta - z| + h(z,\zeta) \big) \, d\nu(\zeta),
\end{align*}
for any $z\in\D$,
where $h(z, \zeta)$ is harmonic in $z$ throughout the unit disc.
In~$K$, the function $u_2$ can be represented as a~sum of 
the logarithmic potential of the finite measure $\nu\bigr|_K$, which is subharmonic by \cite[Theorem 3.1.2]{Ransford},
 and a harmonic function. Therefore $u_2$ is subharmonic in $K$, and consequently in the whole unit disc. By \cite[Theorem 3.7.4]{Ransford} the Riesz measure of~$u_2$ coincides with $\nu/(2\pi)$.

Similar arguments show that functions $u_j^*$ are subharmonic in $\D$ with the Riesz measures 
\begin{equation*}
d\mu_j^*=\frac 1{2\pi} \Delta \varphi \, dm_2 \bigr|_{\bigcup_{n=1}^\infty  Q_{nj}}, \quad j\in \{1,2,3,4\},
\end{equation*}
respectively.
The convergence of the integrals
in~\eqref{e:u_star_def}
follow from the convergence of the sums
\begin{equation*}
\sum_{n=1}^\infty (1-r_n), \quad \sum_{n=1}^\infty (1-\widehat r_n), \quad \sum_{n=1}^\infty (1-r_n'').
\end{equation*}
It then follows from the definition of $u_1$ that 
$\Delta u_1=\Delta \varphi -\Delta u_2 -\sum_{j=1}^4 \Delta u_j^*$, so that 
$\supp \Delta u_1 $ is contained in the closure of the set 
$$B_1=\bigcup_{n=1}^\infty \Big(A_n \cup (\widehat{A}_n \setminus Q_{n1}) \cup ({A}_n'' \setminus Q_{n3}) \Big),$$ 
where the equality $\Delta u_1=\Delta \varphi $ holds. 
We conclude that the Riesz measure of $u_1$ is $\Delta \varphi\,  dm_2 /(2\pi)\bigr|_{\overline B_1}$.

We proceed to approximate the subharmonic functions $u_1$ and $u_2$  using Theorem~\ref{t:3}, and show that 
$|u_1^*(z)|$, $|u_2^*(z)|$, $|u_3^*(z)|$ and $|u_4^*(z)|$ are uniformly small.
We estimate  $u_1^*$ and omit the considerations of $u_2^*$ and $u_3^*$, which are similar. 
We consider three cases. First, let $|\varphi_\zeta(z)|=|(\zeta-z)/(1-\overline\zeta z )|\le 1/2$, 
that is, $\zeta$ belongs to the pseudo-hyperbolic disc $\ph(z, 1/2)$. 
Then
\begin{align*}
u_1^*(z) & = \sum_{n=1}^\infty \int_{Q_{n1}} \log | \varphi_z(\zeta)| \, \Delta \varphi(\zeta)\, dm_2(\zeta)\\
& = \sum_n \int_{Q_{n1} \cap \ph(z,1/2)} \log | \varphi_z(\zeta)| \, \Delta \varphi(\zeta)\, dm_2(\zeta) \\
& \qquad + \sum_{n : \, r_n \leq |z|} \int_{Q_{n1} \setminus \ph(z,1/2)} \log | \varphi_z(\zeta)| \, \Delta \varphi(\zeta)\, dm_2(\zeta)\\
& \qquad + \sum_{n : \, r_n > |z|} \int_{Q_{n1} \setminus \ph(z,1/2)} \log | \varphi_z(\zeta)| \, \Delta \varphi(\zeta)\, dm_2(\zeta) \\
& = S_1(z) + S_2(z) + S_3(z), \quad z\in\D.
\end{align*}
We begin with $S_1$. Since
\begin{align*}
& \left | \int_{Q_{n1} \cap \ph(z,1/2)} \log | \varphi_z(\zeta)| \, \Delta \varphi(\zeta)\, dm_2(\zeta) \right|\\
& \qquad \lesssim \frac{1}{(1-|z|)^2} \int_{\ph(z,1/2)} \log\frac{1}{| \varphi_z(\zeta)|} \, dm_2(\zeta) \\
& \qquad \lesssim \int_{D(0,1/2)} \left(\log\frac{1}{|w|} \right) \frac{1}{|1-\overline{w} z|^4} \, dm_2(w) 
< \infty, \quad z\in\D,
\end{align*}
uniformly for all $n\in\N$,
we deduce
\begin{equation*}
|S_1(z)| \lesssim \# \big\{ n \in\N  : Q_{n1} \cap \ph(z,1/2) \neq \emptyset\big\} < \infty, \quad z\in\D.
\end{equation*}
To estimate $S_2$,
suppose that $|\varphi_z(\zeta)|\ge 1/2$. Then $|\log |\varphi_z(\zeta)| |\le\log 2$ and
	\begin{gather*}
  \left| \, \int\limits_{Q_{n1}\setminus \ph(z, 1/2)}\frac{p_2+\varepsilon_n}{2\pi}\log |\varphi_z(\zeta)| \frac{dm_2(\zeta)}{(1-|\zeta|)^2}\right|			\lesssim\frac{m_2(Q_{n1})}{(1-r_{n})^2} \lesssim 1, \quad n\in\N.
	\end{gather*}
Since
\begin{equation*}
1-r_{n+1} < 1 - r_n' = C \left( \frac{1-r_n}{C} \right)^{(p_2+\varepsilon_n)/p_1} \leq (1-r_n)^{(p_2+\varepsilon_n)/p_1} \leq (1-r_n)^M
\quad n\in\N,
\end{equation*}
for some constant $M>1$
independent of $n$, we conclude $1-r_{n+1} \lesssim (1-r_1)^{M^n}$ for all $n\in\N$.
Therefore
	\begin{equation*} 
 |S_2(z)| \lesssim \sum_{n : \, r_n \leq |z|} 1 = n(|z|,\{r_n\}) \lesssim \log \log \frac{1}{1-|z|}, \quad |z|\to 1^-.
	\end{equation*}
Finally, we estimate $S_3$.
Assume that $|\varphi_z(\zeta)|\ge 1/2$. Since
	\begin{equation*}
	\begin{split}
	0&\le \frac{1}{2} \log \frac{1}{|\varphi_z(\zeta)|^2}
        \leq \frac{1}{2\,  |\varphi_z(\zeta)|^2} \left( 1 - |\varphi_z(\zeta)|^2 \right)
  \le2\, \frac{(1-|z|^2)(1-|\zeta|^2)}{|1-z\overline\zeta|^2}\le 4\, \frac {1-|\zeta|^2}{1-|z|},
	\end{split}
	\end{equation*}
we have
	\begin{gather*}
	\left|\, \int\limits_{Q_{n1}\setminus \ph(z, 1/2)}\frac{p_2+\varepsilon_n}{2\pi}\log |\varphi_z(\zeta)| \frac{dm_2(\zeta)}{(1-|\zeta|)^2}\right|
	\lesssim\frac{1-r_{n}}{1-|z|}, \quad n\in\N, \quad z\in\D,
	\end{gather*}
and therefore
\begin{equation*}
|S_3(z)| \leq \sum_{n: \, r_n> |z|} \frac{1-r_n}{1-|z|}, \quad z\in\D.
\end{equation*}
If $|z|<1/2$, then $S_3(z) \leq 2 \sum_{n=1}^\infty (1-r_n) < \infty$ for all $z\in\D$. If $1/2\leq |z| < 1$, then
\begin{align*}
|S_3(z)| & \leq 1 + (1-|z|)^{M-1} + (1-|z|)^{M^2-1} + \dotsb = 2 \sum_{k=0}^\infty {\left( \frac{1}{2} \right)}^{M^k} < \infty,
\quad z\in\D.
\end{align*}
In conclusion, we have proved
	\begin{equation}\label{e:u1star_est}
	\begin{split}
	|u_1^*(z)|
\lesssim \log \log \frac 1{1-|z|}, \quad |z|\to 1^-.
	\end{split}
	\end{equation}
Similar arguments show
	\begin{equation}\label{e:u_23_star}
	|u_2^*(z)|+ |u_3^*(z)| \lesssim \log \log \frac 1{1-|z|}, \quad |z|\to 1^-.
	\end{equation}

We then consider $u_4$. Now
\begin{align*}
u_4^*(z) & = \sum_{n=1}^\infty \int_{Q_{n4}} \log | \varphi_z(\zeta)| \, \Delta \varphi(\zeta)\, dm_2(\zeta)\\
& = \sum_n \int_{Q_{n4} \cap \ph(z,1/2)} \log | \varphi_z(\zeta)| \, \Delta \varphi(\zeta)\, dm_2(\zeta) \\
& \qquad + \sum_{n : \, \widehat r_n \leq |z|} \int_{Q_{n4} \setminus \ph(z,1/2)} \log | \varphi_z(\zeta)| \, \Delta \varphi(\zeta)\, dm_2(\zeta)\\
& \qquad + \sum_{n : \, \widehat r_n > |z|} \int_{Q_{n4} \setminus \ph(z,1/2)} \log | \varphi_z(\zeta)| \, \Delta \varphi(\zeta)\, dm_2(\zeta) \\
& = T_1(z) + T_2(z) + T_3(z), \quad z\in\D,
\end{align*}
where
$$\Delta \varphi(\zeta) \asymp \frac{\log^2 \frac{C}{1-\widehat{r}_n}}{(1-\widehat r_n)^2}, \quad \zeta \in A_n^*, \quad n\to \infty. $$
We begin with $T_1$. There exists a~constant $c>0$ such that	
\begin{equation*}
	\begin{split}
	&\left|\, \int\limits _{Q_{n4}\cap \ph(z, 1/2)} \log |\varphi_\zeta(z)| \, \frac{\log^2 \frac{C}{1-|\zeta|} 
\, dm_2(\zeta)}{(1-|\zeta|)^2} \right|\\
  &\qquad\lesssim\frac{\log^2 \frac{C}{1-|z|}}{(1-|z|)^2}
	\int\limits_{Q_{n4}\cap \ph(z, 1/2)} \log  \left| \frac{1-\overline{\zeta} z}{\zeta-z} \right| \, dm_2(\zeta)\\
  &\qquad\lesssim\frac{\log^2 \frac{C}{1-|z|}}{(1-|z|)^2}
	\int\limits_{D\left( z,\, c (1-|z|)/(\log \frac{C}{1-|z|}) \right)}
	\log  \left| \frac{1-\overline{\zeta} z}{\zeta-z} \right| dm_2(\zeta)\\
&\qquad\lesssim\frac{\log^2 \frac{C}{1-|z|}}{(1-|z|)^2}
	\int\limits_{D\left( 0,\, c (1-|z|)/(\log \frac{C}{1-|z|}) \right)}
	\log  \left| \frac{1-\overline{(z+w)} z}{w} \right| dm_2(w).
\end{split}
\end{equation*}
Since there exists a~constant $\kappa>1$ such that $|1-\overline{(z+w)} z| \leq \kappa( 1-|z|)$, we conclude
\begin{equation*}
	\begin{split}
	& \left|\, \int\limits _{Q_{n4}\cap \ph(z, 1/2)} \log |\varphi_\zeta(z)| \, \frac{\log^2 \frac{C}{1-|\zeta|} 
\, dm_2(\zeta)}{(1-|\zeta|)^2} \right| \\
& \qquad \lesssim \frac{\log^2 \frac{C}{1-|z|} }{(1-|z|)^2}
	\int\limits_0^{c \frac{1-|z|}{\log \frac{C}{1-|z|}}} \log \frac{\kappa(1-|z|)}{\tau} \, \tau\, d\tau\\
  & \qquad \lesssim\frac{\log^2 \frac{C}{1-|z|} }{(1-|z|)^2}\left(\frac12 \tau^2 \log \frac{\kappa(1-|z|)}{\tau} + \frac{\tau^2}{4}\right) \Bigr|_0^{c \frac{1-|z|}{\log \frac{C}{1-|z|}}}\\
		& \qquad \lesssim\log \log \frac{C}{1-|z|}, \quad |z|\to 1^-,
  \end{split}
	\end{equation*}
uniformly for all $n\in\N$.
We deduce
\begin{equation*}
\begin{split}
|T_1(z)| & \lesssim \# \big\{ n \in\N  : Q_{n4} \cap \ph(z,1/2) \neq \emptyset\big\} \, \log \log \frac{C}{1-|z|} \\
& \lesssim\log \log \frac{C}{1-|z|} , \quad |z|\to 1^-.
\end{split}
\end{equation*}
To estimate $T_2$, suppose that
$|\varphi_\zeta(z)|\ge 1/2$. Then
	\begin{align*}\nonumber
  & \left| \, \int\limits _{Q_{n4}\setminus \ph(z, 1/2)} \log \, \Bigl| \frac{z-\zeta}{1-z\overline\zeta}\Bigr| \, \Delta \varphi(\zeta)\, dm_2(\zeta) \right|  \\ 
& \qquad \lesssim  \int\limits _{Q_{n4}\setminus \ph(z, 1/2)} \big|\log |\varphi_\zeta(z)| \big| \, \frac{\log^2 \frac{C}{1-\widehat{r}_n} \, dm_2(\zeta)}{(1-\widehat r_n)^2} \\
& \qquad  \lesssim    \frac{\log^2 \frac{C}{1-\widehat{r}_n} m_2(Q_{n4})}{(1-\widehat r_{n})^2} \lesssim 1,\quad n\in\N.
\end{align*}
Therefore, as in the case of $S_2$, we deduce
	\begin{equation*} 
 |T_2(z)| \lesssim \sum_{n : \, \widehat r_n \leq |z|} 1 = n(|z|,\{\widehat r_n\}) \lesssim \log \log \frac{1}{1-|z|}, \quad |z|\to 1^-.
	\end{equation*}
Finally, we estimate $T_3$. Assume that $|\varphi_\zeta(z)|\ge 1/2$. Now
\begin{align*}
  \left| \, \int\limits _{Q_{n4}\setminus \ph(z, 1/2)} \log \ \Bigl| \frac{z-\zeta}{1-z\overline\zeta}\Bigr| \, \Delta \varphi(\zeta)\, dm_2(\zeta) \right|   \lesssim \frac{1-\widehat r_{n}}{1-|z|}, \quad  n\in\N.
\end{align*}
Therefore,
\begin{equation*}
|T_3(z)| \leq \sum_{n: \, \widehat r_n> |z|} \frac{1-\widehat r_n}{1-|z|} \lesssim 1, \quad z\in\D.
\end{equation*}
In conclusion, \begin{equation}
         |u_4^*(z)| \lesssim  \log \log \frac1{1-|z|}, \quad |z|\to 1^-. \label{e:u_star_close4}
       \end{equation}

We then approximate $u_1$ and $u_2$ separately.
 The Riesz measure $\mu_1=\mu_{u_1}$  admits a~regular partition
with respect to the function $b_1(z)$:
\begin{enumerate}
\item[\rm (i)]
$\mu_1 (Q_{n,k}^j)=\mu_1(\widehat Q_{n,k}^j)=\mu_1({Q''}_{n,k}^j)=2$ for all $n,k,j$;
\item[\rm (ii)]
the interiors of all polar rectangles  $Q_{n,k}^j$, $\widehat Q_{n,k}^j$ and $ {Q''}_{n,k}^j$ are pairwise disjoint
for all $n,k,j$.
\item[\rm (iii)]
the sides of each rectangle are comparable; 
\item[\rm (iv)]
$\diam Q_{n,k}^j\asymp b_1(r_{n,k})$,
$\diam \widehat Q_{n,k}^j\asymp b_1(\widehat r_{n,k})$ and $\diam  {Q''}_{n,k}^j\asymp b_1(r''_{n,k})$; 
\item[\rm (v)]
$\mu_1\big(\mathbb{D}\setminus \bigcup_n (\bigcup_{k,j}  Q_{n,k}^j \cup
\bigcup_{k,j} \widehat Q_{n,k}^j\cup \bigcup_{k,j}  {Q''}_{n,k}^j) \big)=0$; 
\item[\rm (vi)]
for appropriate values $n, k$, $j$, the measure $\mu_1$ admits the~representation 
$$\mu_1=\sum_n \left(\sum_{k,j} \mu\bigr|_{Q_{n,k}^j}+ \sum_{k,j}\mu\bigr|_{\widehat Q_{n,k}^j}+ \sum_{k,j}\mu\bigr|_{ {Q''}_{n,k}^j}\right).$$
\end{enumerate}
It is easy to check that $\mu_1$ is locally regular with respect to $b_1$, that is, 
$$ \int_0^{b_1(|z|)} \frac
{\mu_1(D(z,t))}t \, dt \lesssim 1, \quad \rho_0<|z|<1,$$ for some constant
$\rho_0\in (0,1)$. In fact, 
we have 
$\mu_1(D(z,t))\lesssim t^2/(1-|z|)^2$ for all $0<t\le \frac{ 1-|z|}2$, which together with 
$b_1(r)=(1-r)/2$ gives local regularity of  $\mu_1$ with respect to $ b_1$.   

By Theorem \ref{t:3} there exists an analytic function $A_1$ in $\D$ such that
\begin{equation} \label{e:u1_oneside}
 \sup_{z\in\D} \big( \log |A_1(z)|- u_1(z) \big)<\infty,
\end{equation}
and for each $\varepsilon>0$ there exist $\rho_1=\rho_1(\varepsilon) \in (0,1)$, $G_1=G_1(\varepsilon)>0$ satisfying
\begin{equation}\label{e:u1_eps_approx}
 \big|\log |A_1(z)|- u_1(z)\big|<G_1, \quad \rho_1<|z|<1, \quad z\not\in E^1_\varepsilon,
\end{equation}
where $E^1_\varepsilon=\{z\in \D: \dist(z, Z_{A_1})\le \varepsilon b_1(|z|)\}$. Moreover, the zero set $Z_{A_1}$ of the function $A_1$ satisfies $Z_{A_1}\subset \bigcup_{\zeta\in \supp \mu} D(\zeta, K_1 \, b_1(|\zeta|))$ for some $K_1>0$.

The Riesz measure $\mu_2=\mu_{u_2}$    admits a~regular partition
with respect to the function $b_2(z)$:
\begin{enumerate}
\item[\rm (i)]
$\mu_2 ({Q^*}_{n}^j)=2$ for all $n,j$; 
\item[\rm (ii)]
the interiors of all polar rectangles  ${Q^*}_{n}^j$ are pairwise disjoint for all $n,j$; 
\item[\rm (iii)]
the sides of each rectangle are comparable; 
\item[\rm (iv)]
$\diam {Q^*}_{n}^j\asymp b_2(r^*_{n})$; 
\item[\rm (v)]
$\mu_2\big(\mathbb{D}\setminus \bigcup_n (\bigcup_{j}  {Q^*}_{n}^j )\big)=0$; 
\item[\rm (vi)]
for all appropriate values $n$, $j$, the measure $\mu_2$ admits the representation 
$$\mu_2=\sum_n \left(\sum_{j} \mu\bigr|_{{Q^*}_{n}^j} \right).$$
\end{enumerate}
Corresponding to the above, the
measure $\mu_2$ is locally regular with respect to $b_2$.
We have 
\begin{equation*}
\mu_2(D(z,t))\lesssim
  \frac{t^2 \big(\log\frac C{1-|z|} \big)^2}{(1-|z|)^2}, \quad 0<t\le \frac{1-|z|}{\log \frac{C}{1-|z|}},
\end{equation*}
which yields the required property.

By Theorem \ref{t:3} there exists an analytic function $A_2$ in $\D$ such that
\begin{equation}\label{e:u2_oneside}
 \sup_{z\in\D} \big( \log |A_2(z)|- u_2(z) \big)<\infty,
\end{equation}
and for each $\varepsilon>0$  there exist $\rho_2=\rho_2(\varepsilon)\in (0,1)$, $G_2=G_2(\varepsilon)>0$ satisfying
\begin{equation}\label{e:u2_eps_approx2}
 \big|\log |A_2(z)|- u_2(z)\big|<G_2, \quad \rho_2<|z|<1, \quad z\not\in E^2_\varepsilon,
\end{equation}
where $E^2_\varepsilon=\{z\in \D: \dist(z, Z_{A_2})\le \varepsilon b_2(|z|)\}$. Moreover, the zero set $Z_{A_2}$ of the function $A_2$ satisfies $Z_{A_2}\subset \bigcup_{\zeta\in \supp \mu} D(\zeta, K_2 \, b_2(|\zeta|))$ for some $K_2>0$.

Define $A=A_1 A_2$. Then, by \eqref{e:u1star_est}, \eqref{e:u_23_star}, \eqref{e:u_star_close4}, \eqref{e:u1_oneside}, and \eqref{e:u2_oneside} we obtain
\begin{align*}
\log|A(z)|-\varphi(z) & =\Big(\log|A_1(z)|-u_1(z) \Big)+\Big(\log|A_2(z)|-u_2(z)\Big)-\sum_{j=1}^4 u_j^*(z) \\
& \lesssim \log \log \frac{1}{1-|z|}, \quad |z|\to 1^-.
  \end{align*}
Since $\varphi(z)=\varphi(|z|)$, there exists a~constant $Y_1>0$ such that
\begin{equation} \label{eq:unoest}
 \log M(r,A)\le \varphi(r)+ Y_1 \log \log \frac{1}{1-|z|}, \quad |z|\to 1^-.
\end{equation}

 Let $E_\varepsilon=E_\varepsilon^1 \cup E_\varepsilon^2$. It is clear that $E_\varepsilon \subset \{z\in \D: \dist(z, Z_A)\le \varepsilon (1-|z|)\}$. Combining \eqref{e:u1star_est}, \eqref{e:u_23_star}, \eqref{e:u_star_close4}, \eqref{e:u1_eps_approx}, and \eqref{e:u2_eps_approx2} we obtain
 \begin{equation*}
   \big|\log|A(z)|-\varphi(z)\big| \lesssim 1+  \log \log \frac{1}{1-|z|}, \quad \max\{\rho_1, \rho_2\}<|z|<1,
\quad  z\not \in E_\varepsilon,
  \end{equation*}
as $|z|\to 1^-$,
where the comparison constant depends on $\varepsilon$. This proves Part~(a) in Lemma~\ref{lemma:approx}.

It follows from the proof of Theorem~\ref{t:3} that for 
each rectangle $Q_{n,k}^j$, $\widehat Q_{n,k}^j$, ${Q''}_{n,k}^j$  
 there corresponds  exactly two zeros of $A_1$ lying in a neighborhood of the rectangle of radius equal to the diameter of the rectangle.
Hence the number of
zeros of $A_1$  in each such neighborhood is uniformly bounded by some constant~$P_1$.
 Suppose that $\partial D(0,r)\cap Q_{n,k}^j\ne \emptyset$ for some $n,k,j$. Then the linear measure of 
$\partial D(0,r) \cap E_\varepsilon^1$ is at most 
a~constant multiple of $P_1\varepsilon$.
  The same is also true  for $\widehat Q_{n,k}^j$, ${Q''}_{n,k}^j$, in addition to $Q_{n,k}^j$.
  Similarly, ${Q^*}_{n}^j$
  corresponds  exactly two zeros of $A_2$ lying in a neighborhood of the rectangle of radius equal to the diameter of the rectangle.
Thus, the number of
zeros of $A_2$  in each such neighborhood is uniformly bounded by some constant $P_2$. Consequently,  the linear measure 
of $\partial D(0,r) \cap E_\varepsilon^2$ is at most 
a~constant multiple of $P_2\varepsilon$,
provided that $\partial D(0,r)\cap {Q^*}_{n}^j\ne \emptyset $.
By fixing a~sufficiently small $\varepsilon>0$,
the linear measure of $ E_\varepsilon  \cap \partial D(0,r) $ is less  than $2\pi r$, where $r\in (1/2,1)$.
 So, if $r$ is sufficiently close to 1,  then there exists $z_r\in \partial D(0,r)\setminus E_\varepsilon$ with $|z_r|=r$.
Therefore,
\begin{equation} \label{eq:dosest}
  \log M(r,A)\ge \log |A(z_r)|\ge \varphi(r)  -  Y_2 \log\log \frac{1}{1-r}, \quad |z|\to 1^-,
\end{equation}
for some $Y_2>0$. By combining~\eqref{eq:unoest} and~\eqref{eq:dosest}, we conclude Part (b) in~Lemma~\ref{lemma:approx}.
The lemma is now proved.
 \end{proof}

\subsection{Completion of the proof of Theorem~\ref{e:sol_limits2}}
The following auxiliary result is stated for convenience.

\begin{oldtheorem}[\protect{\cite[Corollary~6]{CGHRequiv}}] \label{thm:oold}
Let $0<R<\infty$ and $f$
meromorphic in a~domain containing $\overline{D(0,R)}$.
Suppose that $j,k$ are integers with $k>j\geq 0$, and $f^{(j)}\not\equiv 0$.
Then
\begin{align*}
    \int_{r'<|z|<r}  \bigg| \frac{f^{(k)}(z)}{f^{(j)}(z)}  \bigg|^{\frac{1}{k-j}} \, dm_2(z) 
\lesssim R \, \log\frac{e\, (R-r')}{R-r} \, \left(  1 + \log^+ \frac{1}{R-r} + T(R,f)\right)
\end{align*}
for $0\leq r' < r<R$.
\end{oldtheorem}

Let $k\in\N$ be fixed and 
$k\leq p_1<p_2 \leq p < \infty$ be constants such that $p_2 \geq 2k$.
By \cite[Theorem~1.4]{CHR2010}, 
all non-trivial solutions of~\eqref{eq:dek} satisfy $\sigma_M(f)=p_2/k-1$ and therefore it suffices to consider the
lower order of growth. Note that 
in the case $p_2\leq k$, all solutions $f$ of~\eqref{eq:dek} satisfy $\sigma_M(f)=\lambda_M(f)=0$ by \cite[Theorem~1.4]{CHR2010}.

Our argument is based on Lemmas~\ref{lemma:app} and \ref{lemma:approx}, and we take advantage of the same
notation.
Choose $\eta_n=n+1$ for all $n\in\N$ in Lemma~\ref{lemma:app}. As a~general property we note that,
if $0<r<1$ and $1-\rho=C(1-r)^q$ for some $C>0$ and $q>1$, then
	\begin{equation}\label{pmppvk}
	\begin{split}
	\left(\frac{\rho}{r}\right)^\frac1{1-r}
	&=\left(1+\frac{(1-r)\left(1-C(1-r)^{q-1}\right)}{r}\right)^{\frac1{1-r}}\\
	&=e+o(1),\quad r\to1^-.
	\end{split}
	\end{equation}
If $f$ is a solution of $f^{(k)}+Af=0$, where $\log|A|$ is approximated with $\varphi$ as in Lemma~\ref{lemma:approx}, then 
the growth estimate \cite[Theorem~5.1]{HKR:2004} and \eqref{pmppvk} yield
	\begin{align}
	\log^+|f(\widehat{r}_ne^{i\theta})|-C
	&\lesssim\int_0^{\widehat{r}_n}|A(te^{i\theta})|^\frac1k\,dt
	=\left(\int_0^{r_n}+\int_{r_n}^{r_n'}+\int_{r_n'}^{\widehat{r}_n}\right)|A(te^{i\theta})|^\frac1k\,dt \notag\\
	&\lesssim\int_0^{r_n}\frac{dt}{(1-t)^\frac{p_2+o(1)}{k}}
	+\frac1{(1-r_n)^\frac{p_2+\e_{n}+o(1)}{k}}\int_{r_n}^{r_n'}\left(\frac{t}{r_n}\right)^\frac{R_n}{k}\,dt \notag\\
	&\qquad+\int_{r_n'}^{\widehat{r}_n}\left(\frac{t}{r_n}\right)^\frac{R_n}{k}\frac{dt}{(1-t)^{\frac{p_1+o(1)}{k}}}\notag\\
	&\lesssim\frac{1}{(1-r_n)^{\frac{p_2+o(1)}{k}-1}}
	+\frac{r_n'-r_n}{(1-r_n)^\frac{p_2+\e_{n}+o(1)}{k}}
	+\frac1{(1-\widehat{r}_n)^{\frac{p_1+o(1)}{k}-1}}\notag\\
	&\lesssim\frac{1}{(1-\widehat{r}_n)^{\left(\frac{p_1}{p_2+\e_n}\frac{p_2}{p}\right)\left(\frac{p_2+o(1)}{k}-1\right)}}
	+\frac1{(1-\widehat{r}_n)^{\frac{p_1+o(1)}{k}-1}}\notag\\
	&\asymp\frac{1}{(1-\widehat{r}_n)^{\frac{p_1}{p}\left(\frac{p_2}{k}-1\right)+o(1)}}
	+\frac1{(1-\widehat{r}_n)^{\frac{p_1+o(1)}{k}-1}}, \label{eq:ffest}
	\end{align}
for all $\theta\in\R$ as $n\to\infty$, which gives the desired upper bound
for the lower order of growth.

To obtain a~lower bound for the lower order of growth, we argue as follows.
For any $r\in (0,1)$  sufficiently close to $1$ there exists $N\in\N$ such that $r_N < r \leq r_{N+1}$.
Choose $R=(1+r)/2$, $r'=0$, and apply Theorem~\ref{thm:oold} to obtain
\begin{equation*}
 \int_{|z|<r}  \big| A(z) \big|^{1/k} \, dm_2(z)
\lesssim \log\frac{2e}{1-r} \, \left(  1 + \log^+ \frac{2}{1-r} + \log^+M\big( \tfrac{1+r}{2},f\big)\right).
\end{equation*}
Clearly, it suffices to estimate the integral of the coefficient $A$. 
Note also that 
\begin{equation*}
\liminf_{R\to 1^-} \, \frac{\log^+ \log^+ M(R,f)}{\log\frac{1}{1-R}}
 = \liminf_{r\to 1^-} \, \frac{\log^+ \log^+ M(\tfrac{1+r}{2},f)}{\log\frac{1}{1-r}}
\end{equation*}
We continue in two separate parts.

(i) Let $r\geq r_N'$. Then $r_N'\leq r \leq r_{N+1}$. 
Lemma~\ref{lemma:approx} gives a~lower bound for the modulus of the coefficient only
outside the exceptional set. Since the exceptional set lies (almost completely) outside the annulus
$A_n' = \{z\in\D : r_n < |z| < r_n'\}$, we have a~good control of the coefficient there. Note that
$\Delta\varphi$ vanishes on these annuli.
Then, 
\begin{equation*}
\begin{split}
\int_{|z|<r}  \big| A(z) \big|^{1/k} \, dm_2(z) 
& \geq \sum_{n=1}^N \int_{r_n+\varepsilon(1-r_n)<|z|<r_n'-\varepsilon(1-r_n')}  \big| A(z) \big|^{1/k} \, dm_2(z)\\
& \gtrsim \sum_{n=1}^N \frac{1}{(1-r_n')^{\frac{p_1}{k}-\frac{p_1}{p_2}-o(1)}}
 \asymp \frac{1}{(1-r_N')^{\frac{p_1}{k}-\frac{p_1}{p_2}-o(1)}},
\end{split}
\end{equation*}
and therefore
\begin{equation*}
\frac{\log^+ \log^+ M(\tfrac{1+r}{2},f)}{\log\frac{1}{1-r}}
\geq o(1) + \frac{\Big( \frac{p_1}{k}-\frac{p_1}{p_2} - o(1)\Big) \log^+ \frac{1}{1-r_N'}}{\log\frac{1}{1-r_{N+1}}},
\quad r_N'\leq r \leq r_{N+1}.
\end{equation*}
Lemma~\ref{lemma:app} implies
\begin{equation*}
1-r_{N+1} = \frac{1-r_N''}{\eta_N} \sim \frac{1-\widehat r_N}{\eta_N \log\frac{1}{1-\widehat r_N}}
= \frac{(1-r_N')^{\frac{p}{p_2}}}{\eta_N \, \frac{p}{p_2} \log\frac{1}{1-r_N'}}.
\end{equation*}
By the discussion preceding the statement
of Lemma~\ref{lemma:app}, 
\begin{equation*}
\lim_{n\to\infty} \, \frac{\log\eta_n}{\log\frac{1}{1-r_n'}}  = 0,
\end{equation*}
and we obtain
\begin{equation*} 
\frac{\log^+ \log^+ M(\tfrac{1+r}{2},f)}{\log\frac{1}{1-r}}
\geq o(1) + \frac{p_2}{p} \left( \frac{p_1}{k}-\frac{p_1}{p_2} - o(1) \right) ,
\quad r_N'\leq r \leq r_{N+1}.
\end{equation*}

(ii) Let $r< r_N'$. Then $r_N<r<r_N'$. 
Choose $\varepsilon \in (0, \frac{\pi}{C_4})$. It follows from Lemma~\ref{lemma:approx}  that $m_1(\partial D(0,r)\setminus E_\varepsilon)\ge 2\pi -C_4\varepsilon\ge \pi$. 
The, applying Lemma \ref{lemma:app}, we  deduce 
\begin{align*}
\int_{0<|z|<r} |A(z)|^{\frac 1k} dm_2(z) 
& \ge  \int_0^{r_N} \int_{\{\theta\in [0,2\pi]: te^{i\theta}\not \in E_\varepsilon\}} |A(te^{i\theta})|^{\frac 1k}  t\, dt\, d\theta \\ 
& \gtrsim \frac{1}{\log^{C_2/k} \frac1{1-r_N}} \int_0^{r_N} e^{\frac{\varphi(t)}{k}} \, dt \\
& \ge   \frac{1}{\log^{C_2/k} \frac1{1-r_N}} \int_{r_{N-1}''}^{r_N}  
	\frac{ dt}{(1-t)^{\frac{p_2}{k}+o(1)}}  \\ 
& 	\gtrsim \frac{1}{(1-r_N)^{\frac{p_2}{k}-1+o(1)}} = \frac{1}{(1-r_N')^{\frac{p_1}{p_2+\varepsilon_n}(\frac{p_2}{k}-1+o(1))}}\\
&	\ge  \frac{1}{(1-r)^{\frac {p_1}k-\frac{p_1}{p_2}+o(1)}}.
	\end{align*}
Hence,
\begin{equation*} 
\frac{\log^+ \log^+ M(\tfrac{1+r}{2},f)}{\log\frac{1}{1-r}}
\geq o(1) + \frac{p_1}{k}-\frac{p_1}{p_2} 
\geq o(1) + \frac{p_2}{p} \left( \frac{p_1}{k}-\frac{p_1}{p_2} \right) ,
\quad r_N\leq r \leq r_{N}'.
\end{equation*}
By choosing different values for the parameter $p\in [p_2,\infty)$,
we obtain the assertion. Note that the lower order of growth for solutions~\eqref{eq:dek} is as large as possible if $p=p_2$, which corresponds to the
value $\alpha=p_1/p_2$. However, $\alpha=1$ corresponds to the case when both terms in~\eqref{eq:ffest}
are of similar growth.

\section{Growth estimates for logarithm of maximum modulus }\label{sec:maxmod}

This section consists of preparations for the proof of Theorem~\ref{th:first}.
We use  the Wiman-Valiron theory adapted for the lower order of growth.

Recall that $\lambda_*(f)$ and $\sigma_*(f)$ are defined in~\eqref{eq:def_ls} for any function $f$ analytic in~$\D$.
It is known that $\sigma_*(f)=\sigma_M(f)+1$, provided 
$$\limsup_{r\to1^-}\frac{\log M(r,f)}{\log \frac 1{1-r}}=\infty,$$
see~\cite[Lemma 1.2.16]{Str}. The same is true for the maximum term and central index instead of maximum modulus and $K(r,f)$, 
respectively. That is,
	$$
	\limsup_{r\to1^-}\frac{\log^+\nu(r,f)}{\log\frac1{1-r}}
        =\limsup_{r\to1^-}\frac{\log^+\log^+\mu(r,f)}{\log\frac1{1-r}}+1
        = \sigma_M(f)+1,
	$$
see~\cite[Theorems~1.5.1 and~1.5.2]{JK}.
Recall that for an analytic function $f(z)=\sum_{n=0}^\infty \widehat{a}(n)z^n$ in $\D$, 
the maximum term is $\mu(r,f)=\max\{|\widehat{a}(n)r^n| : n\in\N\cup\{0\}\}$, while the central index 
is $\nu(r,f)=\max\{n:\mu(r,f)=|\widehat{a}(n)r^n|\}$. These quantities obey the relation
	\begin{equation} \label{eq:twostars}
	\log\mu(r,f)=\log\mu(r_0,f)+\int_{r_0}^r\frac{\nu(t,f)}{t}\,dt, \quad 0<r_0<r<1,
	\end{equation}
by \cite[Theorem~1.4.1]{JK}, 
assuming that $f(0)=1$ and $\sup\{ |\widehat{a}(n)|: n\in\N\cup\{0\}\}=\infty$.
Note that, if $\sup\{ |\widehat{a}(n)|: n\in\N\cup\{0\}\}<\infty$, then $|f(z)| = O(1/(1-|z|))$
for $z\in\D$.
The representation~\eqref{eq:twostars} implies
\begin{equation} \label{eq:onestar}
\nu(r,f)=r\left(\log\mu(r,f)\right)_+', \quad r_0<r<1.
\end{equation}
In addition, it is known that
	$$
	\lambda_M(f)
	\le\liminf_{r\to 1^-}\frac{\log^+ \nu(r,f)}{\log\frac{1}{1-r}}
	\le\lambda_M(f)+1,
	$$
where both inequalities can be strict~\cite[Theorem~1.5.2]{JK}.

\begin{proposition} \label{p:k_low_ord}
Let $f$ be an analytic function in $\mathbb{D}$. Then $\lambda_*(f)\le\lambda_M(f)+1$ if $f$ is unbounded, 
and $\lambda_M(f)+\frac{\lambda_M(f)}{\sigma_M(f)}\le\lambda_*(f)$ if $\sigma_M(f)>0$.
\end{proposition}

The second part of the statement is similar to a result proved in \cite[Corrigendum]{Sons68} for the maximum term and the central index. But since we have not found a proof in the existing literature, we offer a proof. Proposition~\ref{p:k_low_ord} is a direct consequence of a growth result for convex functions that will be discussed next. For $h:(-\infty,0)\to\mathbb{R}$ such that $\lim_{x\to0^-}h(x)=\infty$, let
	$$
	\alpha(h)=\liminf_{x\to 0^-}\frac{\log h(x)}{\log\frac1{|x|}}\quad\textrm{and}\quad
	\beta(h)=\limsup_{x\to 0^-}\frac{\log h(x)}{\log\frac1{|x|}}.
	$$
Let $\Omega$ denote the class of convex functions $h:(-\infty,0)\to\R$ satisfying the property $\lim_{x\to0^-}h(x)=\infty$. It is easy to see that the right derivative $h'_+$ exists at every point and
	\begin{equation*}
	h(x)=o(h'_+(x)),\quad x\to 0^-,
	\end{equation*}
for each $h\in\Omega$. Hence $\alpha(h)\le\alpha(h'_+)$. In particular, if $\alpha(h)=\infty$, then $\alpha(h'_+)=\infty$.

\begin{proposition}\label{fpr1} Let $h\in\Omega$. Then the following statements are valid:
\begin{enumerate}
\item[(a)] $\alpha(h'_+)\le\alpha(h)+1$;
\item[(b)] $\beta(h'_+)=\beta(h)+1$;
\item[(c)] $\displaystyle\alpha(h)+\frac{\alpha(h)}{\beta(h)}\le\alpha(h'_+)$, if $\alpha(h)<\infty$ and $\beta(h)>0$.
\end{enumerate}
\end{proposition}

\begin{proof}Since $h$ is nondecreasing and positive close to zero, there exists $x_1\in(-\infty,0)$ such that that
	$$
	\left(\frac{x}2-x\right)h'_+(x)
	\le\int_x^{\frac{x}{2}}h'_+(t)\,dt
	=h\left(\frac{x}{2}\right)-h(x)
	\le h\left(\frac{x}{2}\right),\quad x\in(x_1,0).
	$$
Therefore
	$$
	h'_+(x)\le\frac2{|x|}h\left(\frac{x}2\right),\quad x_1<x<0,
	$$
and the inequalities $\alpha(h'_+)\le\alpha(h)+1$ and $\beta(h'_+)\le\beta(h)+1$ follow. Thus, in particular, (a) is proved.

To deduce (b), it remains to show $\beta(h'_+)\ge\beta(h)+1$. Suppose on the contrary that $\beta(h'_+)<\beta(h)+1$. Then there exists $\e>0$ and $x_2\in(-\infty,0)$ such that
	$$
	h'_+(x)\le\left(\frac1{|x|}\right)^{\beta(h)-\e+1},\quad x_2\le x<0,
	$$
and hence
	\begin{align*}
	h(x)-h(x_2)
	&=\int_{x_2}^{x}h'_+(t)\,dt
	\le\int_{x_2}^{x}\left(\frac1{|t|}\right)^{\beta(h)-\e+1}\,dt\\
	&=\frac1{\beta(h)-\e}\left(\left(\frac1{|x|}\right)^{\beta(h)-\e}-\left(\frac1{|x_2|}\right)^{\beta(h)-\e}\right).
	\end{align*}
It follows that $\beta(h)\le\beta(h)-\e$ which is impossible. Thus (b) is proved.

It remains to prove the inequality in (c) under the hypotheses $\alpha(h)<\infty$ and $\beta(h)>0$. If $\alpha(h)=0$ or $\beta(h)=\infty$, the statement is trivial, so assume $\alpha(h)>0$ and $\beta(h)<\infty$. For given $\alpha\in(0,\alpha(h))$ and $\beta\in(\beta(h),\infty)$, there exists $x_3\in(-\infty,0)$ such that $|x|^{-\alpha}\le h(x)\le |x|^{-\beta}$ for all $x_3\le x<0$. Consider the function
	$$
	g(x)=-(2|x|^{\alpha})^{\frac1\beta},\quad -\infty<x<0.
	$$
Since $\alpha<\beta$, we have
	$$
	x-g(x)=-|x|+(2|x|^{\alpha})^{\frac1\beta}\sim(2|x|^{\alpha})^{\frac1\beta},\quad x\to0^-.
	$$
Therefore
	\begin{align*}
	h'_+(x)
	&\ge\frac{h(x)-h(g(x))}{x-g(x)}
	\ge\frac1{x-g(x)}\left(\frac1{|x|^\alpha}-\frac1{|g(x)|^\beta}\right)\\
	&=\frac1{2(x-g(x))}\frac1{|x|^\alpha}
	=\left(\frac1{2^{1+\frac1\beta}}+o(1)\right)\frac1{|x|^{\alpha\left(1+\frac1\beta\right)}},\quad x\to0^-.
	\end{align*}
It follows that
	$$
	\alpha(h'_+)\ge \alpha+\frac{\alpha}{\beta}.
	$$
Since $\alpha\in(0,\alpha(h))$ and $\beta\in(\beta(h),\infty)$ were arbitrary, we deduce (c).
\end{proof}

\begin{proof}[Proof of Proposition \ref{p:k_low_ord}]
Let $h(x)=\log M(e^x,f)$ for all $-\infty<x<0$. Then $h\in\Omega$, $\lambda_M(f)=\alpha(h)$, $\sigma_M(f)=\beta(h)$, $\lambda_*(f)=\alpha(h'_+)$ and $\sigma_*(f)=\beta(h'_+)$. Therefore the inequality $\lambda_*(f)\le\lambda_M(f)+1$ follows by Proposition \ref{fpr1}(a) if $\lambda_M(f)<\infty$, and for otherwise it is trivial.

Assume now $\sigma_M(f)>0$. If $\lambda_M(f)=\infty$, then $\lambda_*(f)=\infty$, and hence the inequality $\lambda_M(f)+\frac{\lambda_M(f)}{\sigma_M(f)}\le\lambda_*(f)$ is clearly valid. But if $\lambda_M(f)<\infty$, then the inequality follows from Proposition~\ref{fpr1}(c).
\end{proof}

The next proposition shows that both estimates for the quantity $\lambda_*(f)$ in Proposition~\ref{p:k_low_ord} are sharp.

\begin{proposition}\label{fpr3} Let $0<\lambda,\sigma<\infty$. Then the following assertions hold.
\begin{enumerate}
	\item[(a)] If $\lambda\le\sigma$, then there exists an analytic function $f$ in $\D$ such that 
$\lambda_M(f)=\lambda$, $\sigma_M(f)=\sigma$ and $\lambda_*(f)=\lambda_M(f)+1$.
	\item[(b)] If $\lambda<\sigma$, then there exists an analytic function $f$ in $\D$ such that $\lambda_M(f)=\lambda$, $\sigma_M(f)=\sigma$ and $\lambda_*(f)=\lambda_M(f)+\frac{\lambda_M(f)}{\sigma_M(f)}$.
\end{enumerate}
\end{proposition}

For the proof we need the following lemma.

\begin{oldlemma}[\protect{\cite{Filevych3}}]\label{flm1} Let $\{n_k\}_{k=0}^\infty$ be an increasing sequence of nonnegative integers and $\{c_k\}_{k=0}^\infty$ an increasing sequence of positive numbers such that $\lim_{k\to\infty}c_k=1$. Let $\{a_n\}_{n=0}^\infty$ be a sequence of complex numbers such that $a_0=\dots=a_{n_0-1}=0$, $a_{n_0}\ne0$,
	$$
	|a_{n_{k+1}}|=|a_{n_0}|\prod_{j=0}^k c_j^{n_{j}-n_{j+1}},\quad k\in\N\cup\{0\},
	$$
and $|a_n|\le|a_{n_k}|c_k^{n_k-n}$ for all $k\in\N\cup\{0\}$ and $n_k<n<n_{k+1}$. Then the power series
$f(z)=\sum_{n=0}^\infty a_nz^n$ represents an analytic function in $\mathbb{D}$ such that
	$$
	\nu(r,f)=\begin{cases}
	n_0, & 0<r<c_0,\\
	n_{k+1}, & c_k\le r<c_{k+1},\quad k\in\N\cup\{0\}.\\
	\end{cases}
	$$
\end{oldlemma}

\begin{proof}[Proof of Proposition \ref{fpr3}] 
First, we prove (a) in the case $\lambda=\sigma$. To do this, we set $n_k=k$ for all $k\in\N\cup\{0\}$, and define
	$$
	c_k=1-\left(\frac{\sigma}{k+\sigma+1}\right)^{\frac1{\sigma+1}},\quad k\in\N\cup\{0\}.
	$$
Then $c_k\in(0,1)$ for all $k$ and $\lim_{k\to\infty}c_k=1$. Further, let $a_0\not=0$ be fixed, and 
$a_{k+1}=a_0\prod_{j=0}^k c_j^{-1}$ for all $k\in\N\cup\{0\}$. Then Lemma~\ref{flm1} ensures that the 
power series $f(z)=\sum_{n=0}^\infty a_nz^n$ represents an analytic function in $\mathbb{D}$ such that $\nu(r,f)=k+1=\frac{\sigma}{(1-c_k)^{\sigma+1}}-\sigma$ for all $c_k\le r<c_{k+1}$ and $k\in\N\cup\{0\}$. It follows that
	$$
	\nu(r,f)\sim\frac{\sigma}{(1-r)^{\sigma+1}},\quad r\to1^-,
	$$
which implies
	$$
	\log\mu(r,f)\sim\frac1{(1-r)^{\sigma}},\quad r\to1^-.
	$$
Since
	$$
	\liminf_{r\to1^-}\frac{\log^+\log^+\mu(r,f)}{\log\frac1{1-r}}
	= \lambda_M(f)
	\le\sigma_M(f)
	=\limsup_{r\to1^-}\frac{\log^+\log^+\mu(r,f)}{\log\frac1{1-r}},
	$$
we deduce $\sigma\le\lambda_M(f)\le\sigma_M(f)=\sigma$. Consequently, Proposition~\ref{p:k_low_ord} yields 
$\lambda_*(f)=\lambda_M(f)+1$. 

We next prove (b). Let $q=\lambda/\sigma\in(0,1)$, fix $\delta\in(0,1)$ such that (further restrictions for $\delta$ apply later)
\begin{equation}\label{feq4}
\frac1{x^{\sigma+1}}+1\le\frac1{x^{\frac{\sigma+1}{q}}},\quad 0<x\le\delta,
\end{equation}
and define $c_k=1-\delta^{q^{-k}}$ for all $k\in\N\cup\{0\}$. Then $c_k\in(0,1)$ for all $k\in\N\cup\{0\}$, and $\lim_{k\to\infty}c_k=1$. Moreover,
	\begin{equation}\label{feq5}
	(1-c_{k+1})^{\lambda+q}=\delta^{\frac{\lambda+q}{q^{k+1}}}=\delta^{\frac{\sigma+1}{q^{k}}}=(1-c_{k})^{\sigma+1},\quad k\in\N\cup\{0\}.
	\end{equation}
Set $n_0=0$ and
	\begin{equation}\label{feq6}
	n_{k+1}=\left[ \frac{1}{\delta^{\frac{\sigma+1}{q^{k}}}} \right]+1,\quad k\in\N\cup\{0\}.
	\end{equation}
Since $\delta^{q^{-k}}\le\delta$, \eqref{feq4} ensures that $\{n_k\}_{k=0}^\infty$ is increasing.
Define $a_0=1$,
	$$
	a_{n_{k+1}}=\prod_{j=0}^k c_j^{n_{j}-n_{j+1}},\quad k\in\N\cup\{0\},
	$$
and $a_n=0$ for all $n\in(n_k,n_{k+1})$ with $k\in\N\cup\{0\}$. Then Lemma~\ref{flm1} implies that the power series $f(z)=\sum_{n=0}^\infty a_nz^n$ represents an analytic function in $\mathbb{D}$ such that $\nu(r,f)=n_{k+1}$ for all $c_k\le r<c_{k+1}$ and $k\in\N\cup\{0\}$. Thus, by \eqref{feq5} and \eqref{feq6},
	$$
	\limsup_{r\to1^-}\frac{\log\nu(r,f)}{\log\frac1{1-r}}=\sigma+1,\quad
	\liminf_{r\to1^-}\frac{\log\nu(r,f)}{\log\frac1{1-r}}=\lambda+q=\lambda+\frac{\lambda}{\sigma}.
	$$
By \cite[Theorem~1.5.2]{JK}, we deduce $\sigma_M(f)=\sigma$.
Using~\eqref{eq:twostars}, we see that the function $h(x)=\log\mu(e^x,f)$ belongs to $\Omega$.
By~\eqref{eq:onestar} and the fact that $\log(1/t) \sim 1-t$ as $t\to 1^-$, we find that
$\alpha(h) = \lambda_M(f)$, $\beta(h)=\sigma_M(f)$ and
\begin{equation*}
\alpha(h_+') = \limsup_{t\to 1^-} \frac{\log \nu(t,f)}{\log\frac{1}{1-t}}.
\end{equation*}
Hence, by Proposition~~\ref{fpr1}(c) and the fact that $\sigma_M(f)=\sigma$, we conclude
\begin{equation*}
\lambda_M(f) + \frac{\lambda_M(f)}{\sigma_M(f)} \leq \alpha(h_+') = \lambda + \frac{\lambda}{\sigma_M(f)},
\end{equation*}
from which $\lambda_M(f) \leq \lambda$. Note that, similarly as above, the identity $\sigma_M(f)=\sigma$ can be proved
alternatively by using Proposition~\ref{fpr1}(b).

We next show that $\lambda_M(f)\ge\lambda$. First observe that $\mu(r,f)\ge\mu(0,f)=a_0=1$ for all $0\le r<1$, and hence
	\begin{align*}
	\log\mu(c_{k},f)& =\log\mu(c_{k-1},f)+\int_{c_{k-1}}^{c_{k}}\frac{\nu(t,f)}{t} \, dt\\
	& \ge\int_{c_{k-1}}^{c_{k}}\frac{\nu(t,f)}{t}\, dt
	=n_{k}\log\frac{c_{k}}{c_{k-1}},\quad k\in\N,
	\end{align*}
where $n_k\sim(1-c_k)^{-q(\sigma+1)}$ as $k\to\infty$ and
	\begin{equation}\label{feq7}
	\begin{split}
	\log\frac{c_{k}}{c_{k-1}}
	&\sim(c_{k}-c_{k-1})
	=(1-c_k)^\frac{\lambda+q}{\sigma+1}-(1-c_k)\\
	&=(1-c_{k})^q(1-(1-c_k)^{1-q})
	\sim(1-c_{k})^q,\quad k\to\infty,
	\end{split}
	\end{equation}
by \eqref{feq5} and \eqref{feq6}. It follows that
	\begin{equation}\label{feq8}
	\log\mu(c_k,f)
	\ge \big(1+o(1)\big ) \frac1{(1-c_{k})^{\lambda}},\quad k\to\infty.
	\end{equation}
By~\eqref{feq8}, we obtain
\begin{align*}
\log\mu(r,f) & = \log\mu(c_k,f) + \int_{c_k}^r \frac{\nu(t,f)}{t} \, dt
 = \log \mu(c_k,f) + n_{k+1} \log\frac{r}{c_k}\\
& \gtrsim \frac1{(1-c_{k})^{\lambda}} + \frac{1}{\delta^\frac{\sigma+1}{q^k}} \, ( r - c_k )
 \gtrsim \frac{1}{(1-r)^\lambda} \, h_k(r),
\quad r\in [c_k,c_{k+1}],
\end{align*}
where
\begin{equation*}
h_k(r) = \left( \frac{1-r}{1-c_k} \right)^\lambda + \frac{(1-r)^\lambda (r-c_k)}{(1-c_k)^{\sigma+1}}.
\end{equation*}
Note that $h_k(c_k) = 1$ and $h_k(c_{k+1}) \sim 1$ as $k\to\infty$. Now $h_k'$ has precisely one zero on $[0,1)$
and $(h_k)'(c_k)>0$ for all $k$ large enough, since $\lambda<\sigma$.
We conclude that $h_k$ is uniformly bounded away from zero on $[c_k,c_{k+1}]$,
and therefore
\begin{align*}
\log\mu(r,f) \gtrsim \frac{1}{(1-r)^\lambda}, \quad r\in [c_k,c_{k+1}].
\end{align*}
As $\bigcup_{k=1}^\infty [c_k, c_{k+1}] = [c_1,1)$, it follows that $\lambda_M(f)\geq \lambda$.
Since we have already shown that $\lambda_M(f)\le\lambda$, we deduce $\lambda_M(f)=\lambda$.

We have shown that $\sigma_M(f)=\sigma$ and $\lambda_M(f)=\lambda$. Therefore the second part of Proposition~\ref{p:k_low_ord}  yields $\lambda+q=\lambda_M(f)+\frac{\lambda_M(f)}{\sigma_M(f)}\le\lambda_*(f)$. Thus, to obtain $\lambda_*(f)=\lambda+q$, it remains to show that $\lambda_*(f)\le\lambda+q$. To do this, let $r_k=2c_k-1=1-2\delta^{q^{-k}}$ for all $k\in\N\cup\{0\}$. Then $r_k\to1^-$ as $k\to\infty$, and
	$$
	\log\frac{c_{k}}{r_k}\sim(c_{k}-r_{k})=1-c_{k}, \quad k\to\infty.
	$$
This together with \eqref{feq5} and \eqref{feq6} yields
	\begin{equation*}
	\begin{split}
	\log n_{k+1}+(n_{k+1}-n_k)\log\frac{r_k}{c_{k}}
	&\sim(\sigma+1)\log\frac1{1-c_k}\\
	&\quad-\left(\frac1{(1-c_k)^{\sigma+1}}-\frac1{(1-c_k)^{q(\sigma+1)}}\right)\left(1-c_k\right)\\
	&\sim-\frac1{(1-c_k)^{\sigma}}=-\frac1{\delta^{\sigma q^{-k}}},\quad k\to\infty.
	\end{split}
	\end{equation*}
Therefore there exists $K_1\in\N$ such that
	\begin{equation}\label{pppl1}
	n_{k+1}\left(\frac{r_k}{c_{k}}\right)^{n_{k+1}-n_k}<\exp\left\{-\frac1{2\delta^{\sigma q^{-k}}}\right\},\quad k\ge K_1.
	\end{equation}
As above, \eqref{feq5}, \eqref{feq6} and \eqref{feq7} yield
	$$
	\log n_{k+1}+(n_{k+1}-n_k)\log\frac{c_{k-1}}{c_{k}}
	\sim-\frac1{(1-c_k)^{\sigma+1-q}}
	=-\frac1{\delta^{(\sigma+1-q)q^{-k}}},\quad k\to\infty.
	$$
Therefore there exists $K_2\in\N$ such that
\begin{equation}\label{pppl2}
\begin{split}
n_{k+1}\left(\frac{c_{k-1}}{c_{k}}\right)^{n_{k+1}-n_k}
&<\exp\left(-\frac1{2\delta^{(\sigma+1-q)q^{-k}}}\right)\\
&<\exp\left(-\frac1{2\delta^{\sigma q^{-k}}}\right),\quad k\ge K_2.
\end{split}
\end{equation}
By applying the trivial inequality $r_k<c_k$, \eqref{pppl1}, \eqref{pppl2}, the fact that the sum
	$$
	\sum_{m=1}^{\infty}\exp\left(-\frac1{\delta^{\sigma q^{1-m}}}\right)
	$$
converges and the trivial estimate $c_k<c_j$ for $j=k+1,\ldots,m-2$, we deduce
	\begin{equation*}
	\begin{split}
	\sum_{m=k+1}^{\infty} n_ma_{n_m}r_k^{n_m}
	&=a_{n_k}r_k^{n_k}\sum_{m=k+1}^{\infty} n_m\frac{a_{n_m}}{a_{n_k}}r_k^{n_m-n_k}\\
	& \leq a_{n_k}r_k^{n_k}\sum_{m=k+1}^{\infty} n_m\prod_{j=k}^{m-1}\frac{r_k^{n_{j+1}-n_j}}{c_j^{n_{j+1}-n_j}}\\
	&\le a_{n_k}r_k^{n_k}\left(n_{k+1}\left(\frac{r_k}{c_{k}}\right)^{n_{k+1}-n_k}+\sum_{m=k+2}^{\infty} n_m
	\left(\frac{c_{m-2}}{c_{m-1}}\right)^{n_{m}-n_{m-1}}\right)\\
	&<a_{n_k}r_k^{n_k}\sum_{m=k+1}^{\infty}\exp\!\left(-\frac1{\delta^{\sigma q^{1-m}}}\right)<a_{n_k}r_k^{n_k},\quad k\ge\max\{K_1,K_2\}.
	\end{split}
	\end{equation*}
The last estimate is justified as follows. If $\delta\in (0,1)$, then
\begin{equation*}
\begin{split}
\sum_{m=1}^{\infty}\exp\left(-\frac1{\delta^{\sigma q^{1-m}}}\right)
 & \leq \sum_{m=1}^\infty \delta^{\frac{\sigma q}{q^m}} = \sum_{m=1}^\infty \exp\left( - \frac{\sigma q}{q^m} \log\frac{1}{\delta} \right)\\
& \leq \frac{1}{\sigma q \log(1/\delta)} \, \sum_{m=1}^\infty q^m
 = \frac{1}{\sigma (1-q) \log(1/\delta)},
\end{split}
\end{equation*}
where the trivial inequality $\exp(-1/x)\leq x$, for $x\geq 0$, has been used two times. Therefore, if 
$\delta<\exp(-1/(\sigma(1-q)))$, then
\begin{equation*}
\sum_{m=1}^{\infty}\exp\left(-\frac1{\delta^{\sigma q^{1-m}}}\right) < 1.
\end{equation*}
Consequently,
	$$
	r_kf'(r_k)=\sum_{m=0}^{k} n_ma_{n_m}r_k^{n_m}+\sum_{m=k+1}^{\infty} n_ma_{n_m}r_k^{n_m}<(n_k+1)\sum_{m=0}^{k} a_{n_m}r_k^{n_m},
	$$
and hence
	$$
	K(r_k,f)=r_k\frac{f'(r_k)}{f(r_k)}
	<\frac{(n_k+1)\sum_{m=0}^{k} a_{n_m}r_k^{n_m}}{\sum_{m=0}^{k} a_{n_m}r_k^{n_m}}=n_k+1,\quad k\ge\max\{K_1,K_2\}.
	$$
This together with~\eqref{feq5} and \eqref{feq6} finally yields
	$$
	\lambda_*(f)
	\le\liminf_{k\to\infty}\frac{\log K(r_k,f)}{\log\frac1{1-r_k}}
	\le\liminf_{k\to\infty}\frac{\log n_k}{\log\frac1{2(1-c_k)}}=\lambda+q=\lambda_M(f)+\frac{\lambda_M(f)}{\sigma_M(f)}.
	$$
Thus $\lambda_*(f)=\lambda_M(f)+\frac{\lambda_M(f)}{\sigma_M(f)}$. This finishes the proof of (b).

To complete the proof of Proposition~\ref{fpr3}, we need to establish (a) in the case $\lambda<\sigma$. But we know by the constructions above that there exist analytic functions $f_1$ and $f_2$ in $\mathbb{D}$ with nonnegative Taylor coefficients such that $\lambda_M(f_1)=\sigma_M(f_1)=\lambda$, $\lambda_*(f_1)=\sigma_*(f_1)=\lambda_M(f_1)+1$, $\lambda_M(f_2)=\lambda$, $\sigma_M(f_2)=\sigma$ and $\lambda_*(f_2)=\lambda+\frac{\lambda}{\sigma}$. The function $f=f_1f_2$ is analytic in $\mathbb{D}$ and satisfies $\lambda_M(f)=\lambda$, $\sigma_M(f)=\sigma$ and  $\lambda_*(f)=\lambda+1$.
\end{proof}

\section{Proof of Theorem~\ref{TheoremLogDeriv}} \label{sec:TLD}

The proof of Theorem \ref{TheoremLogDeriv} is based on an approach from \cite{CHR2010}.
We only prove the difficult case $\lambda_M(f)<\sigma_M(f)$ as the proof of the case $\lambda_M(f) = \sigma_M(f)$
follows from \cite[Corollary~1.3]{CHR2010}, which gives a~logarithmic derivative estimate
that can be extended to a~radial set of upper density one by similar considerations as below.


\subsection{Preparations}\label{prep-sec}

We begin with a growth lemma, which originates from \cite{Linden1956} and associates the order 
of growth with the lower order of growth.
	
\begin{lemma}\label{l:i_est}
Let $f$ be an analytic function in $\D$ such that $0\le\lambda_M(f)<\sigma_M(f)<\infty$. For $1/2\le\alpha<1$, 
$m\in\N\cup\{0\}$ and arbitrary $0\le R_0<1$, define the nondecreasing function $I_{\alpha,m}:[ 0,1)\to[0,\infty)$ by
		\begin{equation*}
		I_{\alpha,m}(R)=\frac{1}{(1-R)^{\frac1\a}}\left(\int_0^R\log^+M\!\left(t,f^{(m)}\right)(R-t)^{\frac1\a-1}\,dt+\log^+M\!\left(R_0,f^{(m)}\right)\right).
		\end{equation*}
Let $\e>0$ and $j\in\N\cup\{0\}$. Then there exist $\alpha=\alpha(\e,f)\in[1/2,1)$ large enough, $\eta=\eta(\e,f)>0$ small enough, and an increasing sequence $\{R_n\}_{n=1}^\infty=\{R_n(\e,f,j)\}_{n=1}^\infty$ of numbers in $(R_0,1)$ tending to 1 such that the following statements are valid:
		\begin{itemize}
			\item[\textnormal{(i)}] $\log M(R_n,f^{(j)})\le (1-R_n)^{-\lambda_M(f)-\frac\eta2}$ for all $n\in\N$;
			\item[\textnormal{(ii)}] The set $E=\bigcup_{n=1}^\infty[R_n^\star, R_n]$, where $(1-R_n^\star)^{\lambda_M(f)+\eta}=(1-R_n)^{\lambda_M(f)+\eta/2}$ for all $n\in\N$, satisfies $\DU(E)=1$ and
			\begin{equation}\label{e:ir_est}
			I_{\alpha,m}(R)\lesssim\frac{1}{(1-R)^{1+\left(\lambda_M(f)-\frac{\lambda_M(f)}{\sigma_M(f)}\right)^++\varepsilon}},\quad R\in E,\quad m=0,\ldots,j.
			\end{equation}
		\end{itemize}
	\end{lemma}
	
\begin{proof} The Cauchy integral formula yields
		$$
		M(r,f^{(k)})\le \frac{k! R\cdot M(R,f)}{(R-r)^k (R+r)}
		\le k!\,\frac{M(R,f)}{(R-r)^k},\quad 0<r<R<1,\quad k\in\N.
		$$
By choosing $R=(1+r)/2$ and taking logarithms we deduce
		$$
		\log M(r,f^{(k)})\le\log M\!\left(\frac{1+r}{2},f\right)+k\log\frac{2}{1-r}+ \log k!, \quad 0<r<1.
		$$
Taking into account the rough estimate
		\begin{equation}\label{relation}
		M(r,f)\le r^k M\!\left(r,f^{(k)}\right)+\sum_{l=0}^{k-1} r^l |f^{(l)}(0)|,\quad 0<r<1,
		\end{equation}
we deduce $\lambda_M(f^{(k)})=\lambda_M(f)$ and $\sigma_M(f^{(k)})=\sigma_M(f)$ for all $k\in\N$.
		
Let $R_0\in[0,1)$ and $j\in\N\cup\{0\}$. Further, let $\alpha\in[1/2,1)$ and $\eta>0$ to be fixed later. By the definition 
of the lower order there exists an increasing sequence $\{R_n\}_{n=1}^\infty=\{R_n(\eta,f,j)\}_{n=1}^\infty$ of
numbers in $(R_0,1)$ such that $\lim_{n\to\infty}R_n=1$ and
		\begin{equation}\label{a1}
		\log M(R_n,f^{(j)})\le (1-R_n)^{-\lambda_M(f)-\frac\eta2}, \quad n\in \mathbb{N}.
		\end{equation}
This proves (i). 

The inequalities \eqref{a1} and \eqref{relation} together yield
		\begin{equation}\label{relation2}
		\log M(R_n,f^{(m)})\le (1-R_n)^{-\lambda_M(f)-\frac\eta2}
		+\log^+\left(\sum_{l=0}^{j-1} |f^{(l)}(0)|\right), \quad n\in \mathbb{N},
		\end{equation}
for all $m\in\N\cup\{0\}$ with $m\le j-1$. Define $\{R_n^\star\}_{n=1}^\infty=\{R_n^\star(R_n)\}_{n=1}^\infty$ by 
$$(1-R_n^\star)^{\lambda_M(f)+\eta}=(1-R_n)^{\lambda_M(f)+\eta/2}, \quad n\in\N,$$
and set $E=\bigcup_{n=1}^\infty[R_n^\star,R_n]$ as in the statement. Then obviously $R_n^\star\to1^-$ as $n\to\infty$, and hence
		\begin{equation} \label{e:dens_e=1}
		1\ge\DU(E)\ge\lim_{n\to\infty}\frac{R_n-R_n^\star}{1-R_n^\star}=1-\lim_{n\to\infty}(1-R_n^\star)^{\frac{\eta}{2\lambda_M(f)+\eta}}=1.
			\end{equation}		
Since $M(R,f^{(j)})$ is nondecreasing, \eqref{a1} and the definition of $R_n^\star$ imply
		\begin{equation*}
		\begin{split}
		\log M\!\left(t,f^{(j)}\right)
		&\le\frac1{(1-R_n)^{\lambda_M(f)+\frac{\eta}{2}}}
		=\frac1{(1-R_n^\star)^{\lambda_M(f)+\eta}}\\
		&\le\frac1{(1-t)^{\lambda_M(f)+\eta}},\quad R_n^\star\le t\le R_n,\quad n\in\N.
		\end{split}
		\end{equation*}
Moreover, \eqref{a1} (the case $j=0$), \eqref{relation2} (the case $j\in\N$) and the definition of $R_n^\star$ yield 
		\begin{equation}\label{e:est_E}
		\log M(t,f^{(m)})\le\frac1{(1-t)^{\lambda_M(f)+\eta}}
		+\log^+\left(\sum_{l=0}^{j-1} |f^{(l)}(0)|\right),\quad R_n^\star\le t\le R_n,\quad n\in\N,
		\end{equation}
for all $m=0,\ldots,j$. Here the logarithmic term disappears for $j=0$.
		
Choose now $\eta=\eta(f,\alpha)>0$ such that
		\begin{equation}\label{Eq.eta}
		\eta<\min\left\{\frac1\alpha-1, \frac{\sigma_M(f)-\lambda_M(f)}{\max \{\sigma_M(f),2\}}\right\}.
		\end{equation}
Let $R\in E$, and define $R^\star$ by the condition $(1-R^\star)^{\sigma_M(f)-\eta}=(1-R)^{\lambda_M(f)+\eta}$. Since $\sigma_M(f)-\eta>\lambda_M(f)+\eta$ by \eqref{Eq.eta}, we have $0<R^\star<R<1$ and $1-R=o(1-R^\star)$, as $R\to 1^-$. By the definition of the order, we have
		$$
		\log^+M(t,f^{(m)})\lesssim(1-t)^{-(\sigma_M(f)+\eta)},\quad 0<r<1,\quad m=0,\ldots,j.
		$$
By using this, the monotonicity of $M\!\left(t,f^{(m)}\right)$ and \eqref{e:est_E}, we deduce
		\begin{equation*}
			\begin{split}
				& I_{\alpha,m}(R)(1-R)^\frac1{\alpha}\\
				& \qquad =\left(\int_0^{R^\star}+\int_{R^\star}^R\right)\log^+M\big(t,f^{(m)}\big)(R-t)^{\frac1\a-1}\,dt+\log^+M(R_0,f^{(m)})\\
				& \qquad \lesssim\int_0^{R^\star}\frac{(R-t)^{\frac1\alpha-1}}{(1-t)^{\sigma_M(f)+\eta}}\,dt
				+\frac1{(1-R)^{\lambda_M(f)+\eta}}\int_{R^\star}^R(R-t)^{\frac1\alpha-1}\,dt+1\\
				& \qquad \le\int_0^{R^\star} \frac1{(1-t)^{\sigma_M(f)+\eta-\frac 1\alpha +1}}\, dt
				+\frac{\alpha(R-R^\star)^{\frac1\alpha}}{(1-R)^{\lambda_M(f)+\eta}}+1\\
				& \qquad \lesssim\frac{{\bf 1}_f}{(1-R^\star)^{\sigma_M(f)+\eta-\frac 1\alpha}}+ \frac{(1-R^\star)^{\frac1\alpha}}{(1-R)^{\lambda_M(f)+\eta}}+1,
				\quad R\in E,
			\end{split}
		\end{equation*}
where ${\bf 1}_f=1$ if $\sigma_M(f)>1$ and zero otherwise. The definition of $R^\star$ now yields
		\begin{equation*}
			\begin{split}
				I_{\alpha,m}(R)
				&\lesssim\frac{{\bf 1}_f}{(1-R)^{\frac1\alpha+\frac{\lambda_M(f)+\eta}{\sigma_M(f)-\eta}\left(\sigma_M(f)+\eta-\frac1\alpha\right)}}
				+\frac1{(1-R)^{\lambda_M(f)+\eta+\frac1\alpha\left(1-\frac{\lambda_M(f)+\eta}{\sigma_M(f)-\eta}\right)}}\\
				&\qquad+\frac1{(1-R)^{\frac1\alpha}},\quad R\in E,\quad m=0,\ldots,j.
			\end{split}
		\end{equation*}
For a given $\varepsilon>0$, choose $\alpha=\alpha(\e,f)\in[1/2,1)$ close enough to 1 and $\eta=\eta(f,\alpha)$ sufficiently small such that $\frac1\alpha<1+\varepsilon$,
		\begin{equation*}
			\begin{split}
				\lambda_M(f)+\eta+\frac1\alpha\left(1-\frac{\lambda_M(f)+\eta}{\sigma_M(f)-\eta}\right)
				<1+\lambda_M(f)-\frac{\lambda_M(f)}{\sigma_M(f)}+\varepsilon
			\end{split}
		\end{equation*}
and
		\begin{equation*}
		\begin{split}
				\frac1\alpha+\frac{\lambda_M(f)+\eta}{\sigma_M(f)-\eta}\left(\sigma_M(f)+\eta-\frac1\alpha\right)
				<1+\lambda_M(f)-\frac{\lambda_M(f)}{\sigma_M(f)}+\varepsilon,\quad \sigma_M(f)>1.
		\end{split}
		\end{equation*}
Since $\eta=\eta(f,\alpha)$ and $\alpha=\alpha(\e,f)$, we deduce (ii).
\end{proof}
	
Let now $n(\zeta,h,f)$ denote the number of zeros of an analytic function $f$ in the 
closed disc $\overline{D}(\zeta,h)=\{w:|\zeta-w|\le h\}$.
Following Hayman~\cite{Hayman1952} and Linden~\cite{Linden1956} we define  
	$$
	u_m(z,h)=\log|f^{(m)}(z)|+N(z,h, f^{(m)}), \quad u(z,h)=u_0(z,h),
	$$
where $N(z,h,f)=\int_{0}^{h}\frac{n(z,t,f)}{t}\, dt$, $h\in(0,1-|z|)$ and $m\in\N\cup\{0\}$. Further, denote $I_\alpha=I_{\alpha,0}$ for short.

By applying \cite[Theorem~2]{Linden1956} to the function $f(Rz)$ at $\zeta/R$, for $|\zeta|<R$, 
we obtain the following result.

\begin{oldtheorem}[\protect{\cite[Theorem~2]{Linden1956}}]\label{TheoremLinden2}
Let $f$ be an~analytic function in $\D$ with $f(0)=1$,
$\a\in[1/2,1)$ and $\widetilde\eta\in(0,1/6)$. 
Then there exist constants $R_0=R_0(\a,f)\in(0,1)$ and $C=C(\a,\widetilde\eta,f)$ such that
		\begin{equation}\label{e:n}
		n(\zeta,h,f)\le\frac{C}{(R-r)^\frac1\a}\left(\int_0^R\log^+M(t,f)(R-t)^{\frac1\a-1}\,dt+\log^+M(R_0,f)\right)
		\end{equation}
and
		\begin{equation}\label{e:u_h_low}
		u(\zeta,h)\ge-\frac{C}{(R-r)^\frac1\a}\left(\int_0^R\log^+M(t,f)(R-t)^{\frac1\a-1}\,dt+\log^+M(R_0,f)\right),
		\end{equation}
where $|\zeta|=r<R$ and $h=\widetilde\eta(R-r)/R$.
\end{oldtheorem}
			
The estimates \eqref{e:n} and \eqref{e:u_h_low}, with $R=\frac{2r}{1+r}$, yield
		\begin{equation}\label{e:nglob}
		n(\zeta,\widetilde\eta(1-r)/2,f)
		\lesssim I_\alpha\!\left(\frac{2r}{1+r}\right), \quad 	
                u(\zeta,\widetilde\eta(1-r)/2) \gtrsim - I_\alpha\!\left(\frac{2r}{1+r}\right)
		\end{equation}
where $|\zeta|=r$ and $0<\widetilde\eta<1/6$. In fact, these estimates can be extended
for larger values of $\widetilde\eta$. We will only consider the extension of the first inequality;
the second extension is similar and hence omitted. Suppose that $\widetilde\eta_0\in(1/6,2)$ is fixed.
The disc $D(\zeta, \widetilde\eta_0(1-|\zeta|)/2)$ can be covered by a~finite number
of discs of the type 
$$D\big(z,\widetilde\eta (1-|z|)/2\big), \quad z\in D\big(\zeta, \widetilde\eta_0(1-|\zeta|)/2\big), $$
where $0<\widetilde\eta<1/6$ is fixed. This property follows from the fact that 
there are $N=N(\widetilde\eta_0, \widetilde\eta)$ smaller discs for which
\begin{equation*}
D\left(z,\widetilde\eta (1-|z|)/2\right)
\supset D\left(z,\frac{\widetilde\eta(2-\widetilde\eta_0)}{4} (1-|\zeta|)\right),
\quad z\in D\big(\zeta, \widetilde\eta_0(1-|\zeta|)/2\big),
\end{equation*}
and where the smaller discs of fixed radii cover the disc $D(\zeta, \widetilde\eta_0(1-|\zeta|)/2)$. Now, by~\eqref{e:nglob}
and the estimate $2r/(1+r)\leq (1+r)/2$, we deduce
\begin{equation} \label{eq:ld}
n\big(\zeta,\widetilde\eta_0(1-r)/2,f\big) 
\lesssim N \cdot I_\alpha\!\left(\frac{1 +  (r+\widetilde\eta_0(1-r)/2)}{2}\right)
\end{equation}
where $|\zeta|=r$.

	 Let $r_\nu=1-2^{-\nu}$ for all $\nu\in\N$. 
	Define $\A_1=\overline{D}(0,1/2)$ and $\A_\nu=\{\zeta:r_{\nu-1}<|\zeta|\le
	r_{\nu}\}$ for all $\nu\in\N\setminus\{1\}$, so that $\D=\bigcup_{\nu=1}^\infty\A_\nu$. 
	
We need an estimate for $$J(z,R):= \int_0^{2\pi} \frac {N(Re^{i\theta}, \frac{1-R}{16},f)}{|Re^{i\theta}-z|^2} \, d\theta.$$

\begin{lemma} \label{l:N_int_est} Let $f$ be an~analytic function in $\D$ with $f(0)=1$.
Then there exists a~constant $C=C(f)>0$ such that
	$$
	\int_{r_{\nu+1}}^{r_{\nu+2}} J(z,R) \, dR \le C \, I_\alpha(r_{\nu+4}),
\quad z\in\A_\nu, \quad \nu\in\N.
	$$ 
\end{lemma}

\begin{proof}
Write $\phi_{\nu,k}=\pi k 2^{-\nu-4}$ and $z_{\nu,k}=(1-3\cdot 2^{-\nu-2})e^{i\pi (2k+1) 2^{-\nu-5}}$ 
for all $k=0,\dots,2^{\nu+5}$. Further, write 
$$\mathcal{R}_{\nu,k}=\big\{w:r_\nu\le|w|\le r_{\nu+1},\,\, \phi_{\nu,k}\le\arg w<\phi_{\nu,k+1}\big\}$$ 
for all $\nu\in\N$ and $k=0,\dots,2^{\nu+5}-1$. Observe that
	$$
	\frac{r_\nu+r_{\nu+1}}{2}=1-3\cdot2^{-\nu-2}\quad\textnormal{and}\quad
	\frac{\phi_{\nu,k}+\phi_{\nu,k+1}}{2}=\pi2^{-\nu-5}(2k+1),
	$$
so $z_{\nu,k}$ is the center of $\mathcal{R}_{\nu,k}$. Then trivially
	\begin{equation}\label{e:r_subset}
	\mathcal{R}_{\nu,k}
	\subset D(z_{\nu,k},2^{-\nu-2}+\pi 2^{-\nu-5})
	\subset D(z_{\nu,k}, 2^{-\nu-1})
	\end{equation}	
and $1-|z_{\nu+1,k}|+2^{-\nu-2}=5\cdot 2^{-\nu-3}$. 

Let $d\mu(\zeta)$ denote the Riesz measure of $\log|f(\zeta)|$, i.e., the counting measure of zeros of $f$.
Then by the definition of $\mathcal{R}_{\nu+1,k}$ together with the monotonicity of $N(w,t,f)$
with respect to $t$, \eqref{e:r_subset} and  Fubini's theorem, we deduce 
	\begin{align*}\label{e:long_est}
	& \int_{\mathcal{R}_{\nu+1,k}} N\left(w,\frac{1-|w|}{16},f\right)\, dm_2(w) \\
	& \qquad \le\int_{D(z_{\nu+1,k}, 2^{-\nu-2})}N\left(w,2^{-\nu-5},f\right)\,dm_2(w)\\
	& \qquad =\int_{D(z_{\nu+1,k}, 2^{-\nu-2})}\left(\int_0^{2^{-\nu-5}} \frac{n(w,t, f)}t\, dt\right)\,  dm_2(w)\\ 
	& \qquad {\color{black} =\int_0^{2^{-\nu-5}}\left(\int _{D(z_{\nu+1,k}, 2^{-\nu-2})}    \frac{n(w,t, f)}t \,  dm_2(w)\right)\, dt}\\ 
	& \qquad =\int_0^{2^{-\nu-5}}\left(\int _{D(z_{\nu+1,k}, 2^{-\nu-2})}\left(\int _{|\zeta-w|\le t} \frac1t  d\mu(\zeta)\right)\,dm_2(w)\right)\,dt\\
	& \qquad \le\int_0^{2^{-\nu-5}}\left(\int _{D(z_{\nu+1,k}, 2^{-\nu-2}+t)}\left(\int_{|w-\zeta|\le t} \frac1t dm_2(w)\right)d\mu(\zeta)\right)\,dt.
\end{align*}
Therefore
\begin{align*}
	 \int_{\mathcal{R}_{\nu+1,k}} N\left(w,\frac{1-|w|}{16},f\right)\, dm_2(w) 
	& \leq\int_0^{2^{-\nu-5}}\left(\int _{D(z_{\nu+1,k}, 2^{-\nu-2}+t)} \pi t d\mu(\zeta)\right)\,dt\\
	& \le\pi \int_0^{2^{-\nu-5}} n\left(z_{\nu+1,k},9 \cdot 2^{-\nu-5},f\right) t\, dt\\
	& =n\left(z_{\nu+1,k}, \frac{3/2}{2} \big( 1 - |z_{\nu+1,k}| \big) ,f\right) \frac{\pi(2^{-\nu-5})^2}{2}.
\end{align*}
By~applying \eqref{eq:ld} for $\widetilde\eta_0=3/2$, we conclude
\begin{equation} \label{eq:ll}
\int_{\mathcal{R}_{\nu+1,k}} N\left(w,\frac{1-|w|}{16},f\right)\, dm_2(w)
\lesssim  I_\alpha(r_{\nu+4}) (1-r_{\nu+1})^2.
\end{equation}

Since $z\in\A_\nu$ by the assumption,
the definitions of $J(z,r)$, $N(t,w,f)$, $\mathcal{R}_{n,k}$
and the estimate~\eqref{eq:ll} yield
		\begin{equation*}
		\begin{split}
		\int_{r_{\nu+1}}^{r_{\nu+2}} J(z,R) \, dR  
		&= 	\int_{r_{\nu+1}}^{r_{\nu+2}} \int_0^{2\pi } \frac {N(Re^{i\theta}, \frac{1-R}{16},f)}{|Re^{i\theta}-z|^2} \, d\theta \, dR\\  
		&\le\frac{1}{r_{\nu+1}}  \sum_ {k=0}^{2^{\nu+5}-1}\int_{\mathcal{R}_{\nu+1,k}} \frac {N(w, \frac{1-|w|}{16},f)}{|w-z|^2} \, dm_2(w)\\
		&\lesssim\sum_ {k=0}^{2^{\nu+5}-1}  \frac 1{|z_{\nu+1,k}-z|^2}\int_{\mathcal{R}_{\nu+1,k}}  {N\left(w, \frac{1-|w|}{16},f\right)} \, dm_2(w)\\ 
		&\lesssim\sum_ {k=0}^{2^{\nu+5}-1}  \frac{I_\alpha(r_{\nu+4})(1-r_{\nu+1})^2}{|z_{\nu+1,k}-z|^2} \\
& \lesssim 			
		I_\alpha(r_{\nu+4})(1-r_{\nu+1}) \int_0^{2\pi} \frac {d\theta}{|z-|z_{\nu+1,0}|e^{i\theta}|^2}\\ 
		&\lesssim I_\alpha(r_{\nu+4}) \, \frac {1-r_{\nu+1}}{|z_{\nu+1,0}|-|z|}\lesssim I_\alpha(r_{\nu+4}),
		\end{split}
		\end{equation*}
which completes the proof of Lemma \ref{l:N_int_est}.
	\end{proof}

	Let $\{a_k\}$ denote the sequence of zeros of $f$ listed
	according to multiplicities and ordered by increasing moduli.
	Let $r\in[r_\nu,r_{\nu+1})$. Then $R=\frac{2r}{1+r}\in(r_{\nu},r_{\nu+2})$. Denote
	\begin{equation} \label{eq:kaak}
	n_1(r)=\max_{\varphi\in[-\pi,\pi)} \, \# \left\{a_k:r\le|a_k|\le\frac{1+r}{2}, \, \, |\arg a_k-\varphi|\le\frac{\pi}{4}(1-r)\right\}.
	\end{equation}
	By choosing $\delta=1/\nu$ in the proof of \cite[Lemma~3.3]{CHR2010}, we deduce that there exists a~countable collection of discs 
        $D_{\nu j}=\{\zeta:|\zeta-z_{\nu j}|<\rho_{\nu j}\}$ with $\rho_{\nu j}<1-|z_{\nu j}|$ such that 
	\begin{equation}\label{e:log_est}
	\sum_{|a_k|\le r_{\nu +1}}\frac 1{|z-a_k|}
	\le24\nu\sum_{s=1}^{\nu+1}\frac{n_1(r_{s-1})}{1-r_{s-1}}+ C \nu
	\sum_{s=\nu-2}^{\nu+1}\frac{n_1(r_s)}{1-r_s}\log n_1(r_s)
	\end{equation}
	for all $z\in\A_\nu\setminus\bigcup_{j}D_{\nu j}$, where
	\begin{equation}\label{e:rad}
	\sum_{R<|z_{\nu j}|<1}\rho_{\nu j}\le\frac{1-R}{-\log(1-R)-1}, \quad R\to 1^-.
	\end{equation}
The polar rectangle in~\eqref{eq:kaak} is of pseudo-hyperbolic diameter strictly less than one, and therefore it can be covered by finitely many pseudo-hyperbolic discs, uniformly for all $0<r<1$.
This allows us to use $\widetilde{\eta}_0=3/2$ in \eqref{eq:ld}.
This inequality
and \eqref{e:log_est}, combined with the monotonicity of $I_\alpha$, then yield
	\begin{equation}\label{e:log_der_pol_est}
	\begin{split}
	\sum_{|a_k|\le r_{\nu+1}}\frac1{|z-a_k|}
	&\lesssim  I_\alpha(r_{\nu+5}) \, \nu \, \Big(\log I_\alpha(r_{\nu+5})+ 1 \Big)\sum_{s=0}^{\nu+1}\frac{1}{1-r_s}\\
	&\asymp\frac{I_\alpha(r_{\nu+5})}{1-r_{\nu+1}} \left( \log\frac{1}{1-r_{\nu+1}} \right) \! \big(\log I_\alpha(r_{\nu+5})+1 \big)
	\end{split}
	\end{equation}
for all $z\in\A_\nu\setminus\bigcup_{j}D_{\nu j}$ and $\nu\in\N$.

	\subsection{Proof of the case $k=1$ and $j=0$}\label{k1j0}
	
	After the preparations in Section~\ref{prep-sec} we are finally ready to prove the special case $k=1$, $j=0$
        and $f(0)=1$. Denote $z=re^{i\varphi}$, where $0<r<R<1$. By the differentiated Poisson-Jensen formula we have
	\begin{equation}\label{EqPoisson-Jensen'}
	\left|\frac {f'(z)}{f(z)}\right|\le\frac R\pi \int_0^{2\pi}
	\frac{|\log |f(Re^{i\theta})||\, d\theta}{|Re^{i\theta} -z|^2 }+2\sum_{|a_k|\leq R}
	\frac1{|z-a_k|}.
	\end{equation}
	By the definition of $u(z,h)$ we have 
	\begin{equation}\label{e:mod_f_est}
          \begin{split}
		|\log|f(w)|| & =\log^+|f(w)|+\log^-|f(w)|\\
& =\log^+|f(w)|+ (N(w,h,f)-u(w,h))^+  \\
		& \le \log^+|f(w)| +N(w,h,f) + u^-(w,h). 
 	\end{split}
        \end{equation}
	
	Let $R_0\in (0,1)$ be as in Theorem~\ref{TheoremLinden2}.
	Denote $p=1+\left(\lambda_M(f)-\lambda_M(f)/\sigma_M(f)\right)^+$ for short. Let $\varepsilon>0$, and let 
$\{R_n\}, \{R_n^\star\} \subset (R_0,1)$ be the sequences in Lemma~\ref{l:i_est}. Let $\nu\in\N$ such that $[r_{\nu+1}, r_{\nu+2}]\subset [R_n^\star, R_n]$ for some $n\in\mathbb{N}$. Such $\nu$ and $n$ exist, and the number of acceptable $\nu$ for given $n$ increases to infinity as $n\to\infty$ because of the identity 
$$(1-R_n^\star)^{\lambda_M(f)+\eta}=(1-R_n)^{\lambda_M(f)+\eta/2}.$$
	Indeed, the hyperbolic distance $\varrho_h(R_n^\star,R_n)$ between the points $R_n^\star<R_n$ increases to infinity, as $n\to\infty$, while the hyperbolic distance $\varrho_h(r_{\nu+1}, r_{\nu+m})$, $m\geq 2$,  tends to the constant value $(m-1)\log(2)/2$, as $\nu\to\infty$.
	
Let	$E=\bigcup_{n=1}^\infty[R_n^\star, R_n]$. Choose $\widetilde \eta=1/8$ and write $h=\widetilde\eta(R-r)/R$
as in Theorem~\ref{TheoremLinden2}.
Since $R \mapsto \log^+ M(R,f)$ is non-decreasing, we obtain
\begin{equation*}
\log^+ M(R,f)
\lesssim \frac{\log^+ M(R,f)}{\big(1-\frac{2R}{R+1}\big)^{1/\alpha}}
\int_R^{\frac{2R}{R+1}} \left( \frac{2R}{R+1} - t \right)^{\frac{1}{\alpha}-1}\, dt
\lesssim I_\alpha \!\left( \frac{2R}{R+1} \right),
\end{equation*}
and therefore the estimate \eqref{e:u_h_low} implies
		\begin{equation}\label{EqPoisson-Jensen}
                  \begin{split}
                    \frac R\pi \int_0^{2\pi}
                    \frac{\log^+ |f(Re^{i\theta})|+u^-(Re^{i\theta},h)}{|Re^{i\theta} -z|^2 }\, d\theta
& \lesssim  I_\alpha \!\left( \frac{2R}{R+1} \right) \int_{0}^{2\pi} \frac{d\theta}{|Re^{i\theta}-z|^2}\\
& \lesssim \frac{I_\alpha(r_{\nu+3})}{R-|z|} \le \frac{I_\alpha(r_{\nu+3})}{1-r_{\nu+1}}
	\end{split}
        \end{equation}
for all $R\in [r_{\nu+1}, r_{\nu+2}]$ and $z\in \mathcal{A}_\nu$.

	Let $E_{\nu+1, 0}=\{R\in [r_{\nu+1}, r_{\nu+2}]: J(z,R) \ge C\nu 2^{\nu+2}I_\alpha (r_{\nu+4})\}$, where
the constant $C$ is as in~Lemma~\ref{l:N_int_est}.
Therefore
\begin{equation}
\label{e:cross}
J(z,R)< \frac{2C\nu\, I_\alpha (r_{\nu+4})}{1-r_{\nu+1}}, \quad R\in [r_{\nu+1}, r_{\nu+2}] \setminus E_{\nu+1,0}, \quad z\in \mathcal{A_\nu}.
\end{equation}
By Chebyshev's inequality and Lemma~\ref{l:N_int_est}, 
\begin{equation*}
 m_1(E_{\nu+1, 0}) 
\le \frac{\int_{r_{\nu+1}}^{r_{\nu+2}} J(z,R) \, dR}{C\nu 2^{\nu+2}I_\alpha (r_{\nu+4})} 
\le \frac{2^{-\nu-2}}{\nu}. 
\end{equation*}

Let $z\in\A_\nu\setminus\bigcup_j D_{\nu j}$
	such that $|z|\in \widetilde{E}_{1,0}$, where $\{D_{\nu j}\}$ is the
	collection of discs mentioned above
	and
	$$
	\widetilde{E}_{1,0}=\bigcup_{n=1}^\infty \left( \bigcup_{[r_{\nu+1}, r_{\nu+5}]\subset [R_n^\star, R_n]} [r_{\nu+1}, r_{\nu+2}] \setminus E_{\nu+1,0}\right).
	$$
	By combining \eqref{e:log_der_pol_est}, \eqref{EqPoisson-Jensen'}, \eqref{e:mod_f_est},  \eqref{EqPoisson-Jensen}, \eqref{e:cross} and Lemma~\ref{l:i_est} we conclude that, if $[r_{\nu+1}, r_{\nu+5}]\subset [R_n^\star, R_n]$, then
	\begin{equation*}
		\begin{split}
			\left|\frac{f'(z)}{f(z)}\right|
			&\lesssim
\frac{I_{\alpha}(r_{\nu+3})}{1-r_{\nu+1}}
+\frac{\nu\, I_{\alpha}(r_{\nu+4})}{1-r_{\nu+1}}
+ \frac{I_\alpha(r_{\nu+5})}{1-r_{\nu+1}} \left( \log\frac{1}{1-r_{\nu+1}} \right) \! \big(\log I_\alpha(r_{\nu+5})+1\big)\\
			&\lesssim\frac1{(1-r_{\nu+5})^{p+1+ 2\varepsilon}}
			\lesssim\frac1{(1-r_{\nu-1})^{p+1+ 2\varepsilon}}\\
			& \leq  \left(\frac{1}{1-|z|}\right)^{2+\left(\lambda_M(f)-\frac{\lambda_M(f)}{\sigma_M(f)}\right)^++2\varepsilon}.
		\end{split}
	\end{equation*}
To prove $\DU(\widetilde{E}_{1,0}) = 1$, note that
\begin{equation*}
\begin{split}
m_1\big([r_{\nu+1}, r_{\nu+2}] \setminus E_{\nu+1,0} \big) & = m_1 \big([r_{\nu+1}, r_{\nu+2}] \big) - m_1\big(  E_{\nu+1,0}\big)\\
& \geq 2^{-\nu-2} - \nu^{-1} \, 2^{-\nu-2} = (1-1/\nu)\,  m_1\big([r_{\nu+1}, r_{\nu+2}] \big),
\end{split}
\end{equation*}
and therefore, in view of \eqref{e:dens_e=1}, 
\begin{align*}
\DU(\widetilde{E}_{1,0}) & {\color{black} = \limsup_{r\to 1^-} \frac{m_1(\widetilde{E}_{1,0} \cap [r,1))}{1-r}}\\
& {\color{black} \geq \lim_{n\to \infty} \frac{1}{1-R_n^\star} \, \sum_{[r_{\nu+1}, r_{\nu+5}]\subset [R_n^\star, R_n]} m_1\big([r_{\nu+1}, r_{\nu+2}] \setminus E_{\nu+1,0} \big)}\\
 & \geq  \lim_{n\to \infty} \frac{1}{1-R_n^\star} \, \sum_{[r_{\nu+1}, r_{\nu+5}]\subset [R_n^\star, R_n]} (1-1/\nu) \, m_1\big([r_{\nu+1}, r_{\nu+2}] \big)\\
& \geq \lim_{n\to \infty} \frac{\big( R_n - R_n^\star\big) \big( 1 - o(1) \big)}{1-R_n^\star} 
=  1.
\end{align*}
We have proved the estimate~\eqref{EqLogEstimate>1} for the radial set
\begin{equation*}
\widetilde{E}_{1,0} \setminus \left\{ r \in [0,1) : \text{$re^{i\theta} \in \textstyle \D \setminus \bigcup_\nu \bigcup_j D_{\nu j}$ 
for all $e^{i\theta}\in\partial\D$} \right\},
\end{equation*}
which is of upper density one since
$\DU(\widetilde{E}_{1,0})=1$ and the excluded set
is of upper density zero by~\eqref{e:rad}.

	This completes the proof of Theorem~\ref{TheoremLogDeriv} in the case $k=1$, $j=0$ and $f(0)=1$.
If $f(0)\neq 1$, then there exist $K\in\C\setminus\{0\}$ and $q\in\N\cup\{0\}$ such that $g(z)=Kz^{-q}f(z)$ is analytic
in $\D$ and $g(0)=1$. By applying the argument above to $g$, we conclude the assertion in the case $k=1$ and $j=0$,
with constants depending on $f$.

	\subsection{Proof of the general case $k>j\ge 0$}
	
	Suppose now that $k>j\geq 0$.
	In Section~\ref{prep-sec} we proved that
	$\lambda_M(f^{(m)})=\lambda_M(f)$ and $\sigma_M(f^{(m)})=\sigma_M(f)$
	for all $m\in\N$. Hence the constant $p$ in Section~\ref{k1j0}
	is the same for all derivatives $f^{(m)}$.
	
	We apply the reasoning in the case $k=1$ and $j=0$ to the functions $f^{(m)}$, where
	$m=j,\ldots,k-1$. Since the upper bound for $I_{\alpha,m}(R)$ in \eqref{e:ir_est} is
	uniform for $m=j,\ldots,k-1$,
        each derivative
	$f^{(m)}$ is associated with 
	the same radial set~$E = \bigcup_{n=1}^\infty [R_n^\star,R_n]$ 
        and an~individual countable collection of discs $D_\nu^{(m)}=D(z_{\nu}^{(m)}, r_{\nu}^{(m)})$
	satisfying $r_\nu^{(m)}<1-|z_\nu^{(m)}|$ for all $\nu$ such that
	\begin{equation} \label{eq:zero}
	\sum_{R\leq |z_{\nu}^{(m)}|<1}\rho_{\nu}^{(m)}\leq\frac{1-R}{-\log(1-R)-1},\quad R\to 1^-.
	\end{equation}
	Furthermore, we deduce the estimate
	$$
	\left|\frac{f^{(m+1)}(z)}{f^{(m)}(z)}\right|\lesssim\frac{1}{(1-|z|)^{p+1+2\varepsilon    }},
	\quad z\in\D\setminus\textstyle\bigcup_{\nu} D_\nu^{(m)},\quad |z|\in \widetilde{E}_{m+1,m},
	$$
where
$$
	\widetilde{E}_{m+1,m}
        =\bigcup_{n=1}^\infty \left( \bigcup_{[r_{\nu+1}, r_{\nu+5}]\subset [R_n^\star, R_n]} [r_{\nu+1}, r_{\nu+2}] \setminus E_{\nu+1,m}\right).
	$$
This argument can be repeated for all $m=j,\ldots,k-1$.
The estimate for the generalized logarithmic derivative follows by writing
\begin{equation*}
\left| \frac{f^{(k)}(z)}{f^{(j)}(z)} \right| 
= \left| \frac{f^{(k)}(z)}{f^{(k-1)}(z)} \right| \dotsb \left| \frac{f^{(j+1)}(z)}{f^{(j)}(z)} \right|.
\end{equation*}
Finally, the logarithmic derivative estimate~\eqref{EqLogEstimate>1} holds on the radial set
\begin{equation} \label{eq:radialset}
\widetilde{E}_{k,j} \setminus \left\{ r \in [0,1) : \text{$re^{i\theta} \in \textstyle \D \setminus \bigcup_{m=j}^{k-1} \bigcup_\nu D_{\nu}^{(m)}$ 
for all $e^{i\theta}\in\partial\D$} \right\},
\end{equation}
where
$$
	\widetilde{E}_{k,j}
        =\bigcup_{n=1}^\infty \left( \bigcup_{[r_{\nu+1}, r_{\nu+5}]\subset [R_n^\star, R_n]} [r_{\nu+1}, r_{\nu+2}] 
\setminus \Big( E_{\nu+1,k-1} \cup \dotsb \cup E_{\nu+1,j} \Big) \right).
	$$
The radial set~\eqref{eq:radialset} is of upper density one since $\DU(\widetilde{E}_{k,j})=1$ and 
the set excluded in~\eqref{eq:radialset} is of upper density zero by~\eqref{eq:zero}.

\section{Proof of Theorem~\ref{th:first}} \label{sec:th:first}

(a) Let $f$ be a non-trivial solution of \eqref{eq:dek}. If $\sigma_M(f)=\lambda_M(f)$, then the assertion follows from \cite[Theorem~1.4(a)]{CHR2010}. Assume next $\lambda_M(f)<\sigma_M(f)$, and apply Theorem~\ref{TheoremLogDeriv} to $f$. 
There exists a~set $E\subset [0,1)$ of $\DU(E)=1$ and a~constant $C>0$ such that
	$$
	M(r,A)= M\bigg(r,\frac{f^{(k)}}{f}\bigg)
	\le C\left(\frac{1}{ (1-r)^{2+\left(\lambda_M(f)-\frac{\lambda_M(f)}{\sigma_M(f)}\right)^++\varepsilon}}\right)^{k},\quad r\in E.
	$$
Therefore $p_1\le k\Big(2+\Big(\lambda_M(f)-\frac{\lambda_M(f)}{\sigma_M(f)}\Big)^+\Big)$, which is equivalent to (a).

(b) 
We consider the case $p_1<p_2$ only. The case $p_1=p_2$ follows from Theorem~\ref{thm:new},
which in turn will be proved in Section~\ref{sec:proofnew} by Theorem~\ref{th:first}(a).
Assume $2k<k\big(2+\tfrac{p_2-2k}{p_2}\big)<p_1<p_2<\infty$, and let $f$ be a non-trivial solution of \eqref{eq:dek}. Then $\sigma_M(f)=\frac{p_2}{k}-1>1$ by \cite[Theorem~1.4]{CHR2010}, and hence $\sigma_*(f)=\frac{p_2}{k}>2$ by~\cite[Lemma~1.2.16]{Str}. Further, Proposition~\ref{p:k_low_ord} yields
	$$
	\liminf_{r\to1^-}\frac{\log K(r,f)}{\log\frac1{1-r}}=\lambda_*(f)\ge\lambda_M(f)+\frac{\lambda_M(f)}{\sigma_M(f)},
	$$
which together with \eqref{eq:th12''} and the assumptions on $p_1$ and $p_2$ imply
	$$
	\lambda_*(f)\ge\frac{\left(\frac{p_1}{k}-2\right)\left(1+\frac1{\sigma_M(f)}\right)}{1-\frac1{\sigma_M(f)}}
	=\frac{\left(\frac{p_1}{k}-2\right)\frac{p_2}k}{\frac{p_2}{k}-2}
	=\frac{p_2(p_1-2k)}{k(p_2-2k)}>1.
	$$
To proceed, we need the following result due to Strelitz.

\begin{oldtheorem}[\protect{\cite[Theorem~1.4.25, p.~282]{Str}}]\label{t:Strelitz} Let $f$ be an analytic function in $\mathbb{D}$ such that $(1-r)K(r,f)\to\infty$ as $r\to 1^-$. Then
	$$
	f^{(n)}(z)\sim\left(\frac{K(|z|,f)}z\right)^nf(z),\quad |z|\to1^-,
	$$
holds for all $|z|\in F=[0,1)\setminus E$ and $z\in\big\{\zeta:|f(\zeta)|\ge K(|\zeta|,f)^{-\beta(|\zeta|)}M(|\zeta|,f)\big\}$, provided one of the following conditions hold:
\begin{itemize}
\item[\rm (i)] $\lambda_*(f)>1$ and $\beta(r)\le q<\frac{\lambda_*(f)-1}{2\lambda_*(f)}$
with $E$ being of finite logarithmic measure;
\item[\rm (ii)] $\sigma_*(f)>1$ and $\beta(r)\le q<\frac{\sigma_*(f)-1}{2\sigma_*(f)}$ with $F$ being of infinite logarithmic measure.
\end{itemize}
\end{oldtheorem}

Since $\lambda_*(f)>1$, by Theorem~\ref{t:Strelitz}(i), for $0\le\beta<(\lambda_*(f) -1)/(2\lambda_*(f))$ there exists a set $E\subset[0,1)$ of finite logarithmic measure such that for $r\in[0,1)\setminus E$, $z$ with $|z|=r$ and $|f(z)|\ge M(r,f)K(r,f)^{-\beta}$ we have
	$$
	\frac{f^{(k)}(z)}{f(z)}\sim\frac{K(r,f)^k}{z^k}.
	$$
Hence, for $\varepsilon>0$ and for those $r=|z|$,
	$$
	M(r,A)\ge|A(z)|=\bigg|\frac{f^{(k)}(z)}{f(z)}\bigg|
	\sim K(r,f)^k
	\ge\frac{1}{(1-r)^{(\lambda_*(f)-\varepsilon)k}},\quad r\to1^-.
	$$

By the assumption $\lambda_{M,\log}(A)=p_1$,
there exists an~increasing sequence $\{s_n\}$, tending to $1$, such that 
$M(s_n,A)\leq 1/(1-s_n)^{p_1+\varepsilon}$ for all $n$. For any $r\in [s_n-\varepsilon(1-s_n),s_n]$
and for all $n$,
we have
\begin{equation*}
M(r,A) \leq M(s_n,A) 
\leq \left( \frac{1+\varepsilon}{1-r} \right)^{p_1+\varepsilon}.
\end{equation*}
The set $S=\bigcup_n [s_n-\varepsilon(1-s_n),s_n]$ is clearly of infinite logarithmic measure.
We conclude that there exists a~sequence $\{r_n\} \subset F \cap S$, tending to $1$, and therefore
	$$
	\frac{1-o(1)}{(1-r_n)^{(\lambda_*(f)-\varepsilon)k}}\leq \left( \frac{1+\varepsilon}{1-r_n} \right)^{p_1+\varepsilon},\quad n\to\infty,
	$$
for all $\e>0$. It follows that $\lambda_*(f)\le\frac{p_1}k$.
Applying Proposition~\ref{p:k_low_ord} we finally deduce $\lambda_M(f)+\frac{\lambda_M(f)}{\sigma_M(f)}\le\lambda_*(f)\le\frac{p_1}k$.

\section{Proof of Theorem~\ref{thm:new}} \label{sec:proofnew}

In order to prove Theorem~\ref{thm:new}, we need Theorem~\ref{th:first}(a)
and Proposition~\ref{prop:extra-de} below. The latter result
is parallel to Theorem~\ref{th:first}(b), 
but has less a~priori assumptions on the
parameters $p_1$, $p_2$ and $k$.
A~straightforward computation shows that the estimate in Proposition~\ref{prop:extra-de}
is actually weaker than the estimate in Theorem~\ref{th:first}(b),
but its value stems from its weaker hypothesis.

\begin{proposition}\label{prop:extra-de}
Let $k\in\N$ and let $A$ be an~analytic function in $\D$ such that
$\sigma_{M,\deg}(A)=p_2>2k$ and $\lambda_{M,\deg}(A)=p_1$.
Then, all nontrivial solutions $f$ of \eqref{eq:dek} satisfy
\begin{equation*} 
\lambda_M(f) \leq \frac{\xi}{k} - 1 < \frac{p_2}{k} - 1 = \sigma_M(f),
\end{equation*}
where the constant $\xi = (1/2) \big( k + \sqrt{k^2 + 4 p_1 (p_2-k)} \big)$ belongs to $(p_1,p_2)$.
\end{proposition}

\begin{proof}
Suppose that $\varepsilon>0$ satisfies $\varepsilon < (p_2-p_1)(p_2-k) p_1^{-1}$.
It is easy to see that $\alpha=\xi_\varepsilon$, where
\begin{equation*} 
\xi_\varepsilon = \frac{ k + \sqrt{k^2 +4 p_1 (p_2 + \varepsilon - k)}}{2},
\end{equation*}
is a solution of $h(\alpha) = \alpha^2 - k \alpha - p_1(p_2 + \varepsilon - k) = 0$. 
It is immediate that $\xi_\varepsilon > \xi_0 >k$ for all $\varepsilon>0$.
Since $h$ is strictly increasing
for all $\alpha>k/2$, $h(p_1) < 0$ and $h(p_2)>0$, we conclude that $\xi_\varepsilon \in (p_1,p_2)$.
For sufficiently small $\varepsilon>0$, we may assume that $\xi_\varepsilon+\varepsilon < p_2$.
Define
\begin{equation*}
\beta = \frac{(\xi_\varepsilon + \varepsilon) (\xi_\varepsilon + \varepsilon -k)}{p_2+\varepsilon-k}.
\end{equation*}
Consequently, $\beta>p_1$ if and only if $h(\xi_\varepsilon + \varepsilon)>0$, and hence we obtain inequalities
$p_1 < \beta < \xi_\varepsilon + \varepsilon < p_2$. Moreover,
\begin{equation} \label{eq:beq}
\frac{1}{\beta} \left( \frac{\xi_\varepsilon + \varepsilon}{k} - 1 \right) = \frac{1}{\xi_\varepsilon + \varepsilon} \left( \frac{p_2+ \varepsilon}{k} - 1 \right).
\end{equation}
Now, there exist sequences $(r_n^*)$ and $(r_n)$, satisfying $r_n^*\to 1^-$ and $r_n \to 1^-$ as $n\to\infty$,
such that $0 < r_n^* < r_n < r_{n+1}^* < r_{n+1} < 1$,
\begin{equation*}
\frac{\log^+ M(r_n^*,A)}{\log\frac{1}{1-r_n^*}} = \xi_\varepsilon + \varepsilon,
\quad 
\frac{\log^+ M(r_n,A)}{\log\frac{1}{1-r_n}} = \beta
\quad \text{and} \quad
\frac{\log^+ M(t,A)}{\log\frac{1}{1-t}} < \xi_\varepsilon + \varepsilon
\end{equation*}
for all $t\in (r_n^*,r_n]$. Now
\begin{equation} \label{eq:rn}
1 - r_n^* = \frac{1}{M(r_n^*, A)^{1/(\xi_\varepsilon+\varepsilon)}}
\quad \text{and} \quad
1 - r_n = \frac{1}{M(r_n, A)^{1/\beta}}.
\end{equation}
By means of \eqref{eq:beq} and \eqref{eq:rn}, we get
\begin{equation} \label{eq:nice}
\frac{(1-r_n)^{(\xi_\varepsilon+\varepsilon)/k - 1}}{(1-r_n^*)^{(p_2+\varepsilon)/k - 1}}
          = \frac{\left( M(r_n^*, A)^{1/(\xi_\varepsilon + \varepsilon)} \right)^{(p_2+\varepsilon)/k - 1}}{\left( M(r_n, A)^{1/\beta} \right)^{(\xi_\varepsilon + \varepsilon)/k - 1}}
           \leq 1. 
\end{equation}

Let $f$ be a~non-trivial solution of \eqref{eq:dek}. Since $\sigma_{M,\deg}(A)=p_2>2k$,
\cite[Theorem~1.4]{CHR2010} implies that $\sigma_M(f) = p_2/k -1$.
Then, for $n\in\N$ large enough, we deduce from
the growth estimate~\cite[Theorem~5.1]{HKR:2004} that
\begin{align*}
  M(r_n,f)
  & \lesssim  \exp \!\left( k \int_0^{r_n^*} M(s,A)^{\frac{1}{k}} \, ds + k \int_{r_n^*}^{r_n} M(s,A)^{\frac{1}{k}} \, ds\right)\\
  & \lesssim \exp \!\left( k \int_0^{r_n^*} \frac{C}{(1-s)^{(p_2+\varepsilon)/k}} \, ds + k \int_{r_n^*}^{r_n} \frac{1}{(1-s)^{(\xi_\varepsilon+\varepsilon)/k}} \, ds\right)\\
  & \lesssim \exp \!\left( \frac{C k^2}{(p_2+\varepsilon-k)(1-r_n^*)^{(p_2+\varepsilon)/k-1}}
     + \frac{k^2}{(\xi_\varepsilon+\varepsilon-k)(1-r_n)^{(\xi_\varepsilon+\varepsilon)/k-1}} \right),
\end{align*}
where $C>0$ is a~constant.
Therefore, by taking~\eqref{eq:nice} into account, we get
\begin{equation*}
\sup_{n\in\N} \, (1-r_n)^{(\xi_\varepsilon+\varepsilon)/k-1} \, \log^+ M(r_n, f) < \infty.
\end{equation*}
Consequently, by letting $\varepsilon\to 0^+$,
we see that $\lambda_M(f) \leq \frac {\xi}k-1 < \sigma_M(f)$, where
\begin{equation*}
\xi = \lim_{\varepsilon\to 0^+} \left( \xi_\varepsilon + \varepsilon \right)
      = \frac{ k + \sqrt{k^2 +4 p_1 (p_2 - k)}}{2}.
\end{equation*}
This completes the proof of Proposition~\ref{prop:extra-de}.
\end{proof}


With these preparations, we are finally ready to present the
proof of 
Theorem~\ref{thm:new}.
Denote $\sigma_{M,\deg} (A)=p_2$ and $\lambda_{M,\deg} (A)=p_1$, for short.

Assume $p_2=p_1=p>2k$. Then $\sigma_M(f) = p/k -1$ for any non-trivial solution $f$ of~\eqref{eq:dek}
by \cite[Theorem~1.4]{CHR2010}. By Theorem~\ref{th:first}(a), $\lambda_M(f) \geq p/k-1 = \sigma_M(f)$.
Therefore $\sigma_M(f)=\lambda_M(f)=p/k-1>1$ for any non-trivial solution $f$ of~\eqref{eq:dek}.

Conversely,
assume $\sigma_M(f)=\lambda_M(f)=p/k-1>1$ for some non-trivial solution $f$ of~\eqref{eq:dek}.
By \cite[Theorem~1.4]{CHR2010}, we conclude $p_2=k(\sigma_M(f)+1) = p$. Suppose on the contrary to the assertion
that $p_1<p_2$. Then, Proposition~\ref{prop:extra-de} implies that all non-trivial solutions of~\eqref{eq:dek}
satisfy $\sigma_M(f) \leq \xi/k-1 < p/k-1 = \sigma_M(f)$, which is a~contradiction. This proves $p_2=p_1=p>2k$.

\section{Example} \label{sec:example}

The following example addresses the case when  the condition $\lambda_{M,\deg}(A)>2k$ is not satisfied. It seems that then  the correlations between 
the growth indicators of the coefficient 
and the growth indicators of solutions 
of~\eqref{eq:dek}
become even more complicated. The reasoning below illustrates this situation for nonvanishing solutions.

Let $\psi$ belong to the class ${\rm BV} [-\pi, \pi]$ of complex-valued 
functions of bounded variation, and let
$$ \omega(\delta, \psi)=\sup \big\{ |\psi(x)-\psi(y)|: \, |x-y|<\delta, \, \, x,y\in [-\pi,\pi]\big\}$$
be the modulus of continuity of $\psi$. 
For $\gamma\in (0,1]$, let $\Lambda_\gamma$ be the class of functions $\psi$ for which $\omega(\delta, \psi) \lesssim \delta^\gamma$ as 
$\delta\to 0^+$ \cite{Z}.

First, let  $0<\varkappa_1<\varkappa_2< 1$ and $\varkappa_1<\alpha$. By \cite[Theorem~6]{ChBer19} there exists an~analytic function $h_\alpha$ in $\D $ of the form
\begin{equation*}
h_\alpha(z)=\int_{-\pi}^{\pi} \frac{d\psi(t)}{(1-ze^{-it})^\alpha}, \quad z\in\D,
\end{equation*}
where $\psi$ is nondecreasing and satisfies
  $\omega(\delta_n, \psi)=O(\delta_n^{\varkappa_2})$ for some sequence $\{\delta_n\}$ tending to zero, 
while $\psi\in \Lambda_{\varkappa_1}$,   $\sigma_{M,\log}(h_\alpha)=\alpha-\varkappa_1$, and
\begin{equation} \label{e:t6}
 \lambda_{M,\log}(h_\alpha)=\begin{cases} 
 \frac{\alpha(\alpha-\varkappa_1)(1-\varkappa_2)}{\alpha(1-\varkappa_2)+\varkappa_2-\varkappa_1},& \alpha<1, \\
                                                                 \alpha-\varkappa_2   , &  \alpha> 1.
                                                                   \end{cases}
\end{equation}
It follows from the construction of $h_\alpha$ that, for $\alpha\in (0,1)$,
\begin{equation}\label{e:re_f_a_est}
{\rm Re} \, h_\alpha (r)\ge  \left( \frac{\alpha(\alpha-\varkappa_1)(1-\varkappa_2)}{\alpha(1-\varkappa_2)+\varkappa_2-\varkappa_1}+o(1)\right) \log \frac{1}{1-r}, \quad r\to 1^-.
\end{equation}
We define $f(z)=e^{h_\alpha(z)}$, $\alpha \in (0,1)$. 
Since
\begin{equation*}
{\rm Re }\, \frac{1}{(1-ze^{-it})^\alpha} \asymp \frac{1}{|1-ze^{-it}|^\alpha},
\end{equation*}
the conditions \eqref{e:t6} and \eqref{e:re_f_a_est} imply
\begin{equation*}
\sigma_M(f)=\alpha-\varkappa_1, \quad \lambda_M(f)= \frac{\alpha(\alpha-\varkappa_1)(1-\varkappa_2)}{\alpha(1-\varkappa_2)+\varkappa_2-\varkappa_1}.
\end{equation*}
Direct computation shows that $\lambda_M(f) \in (\alpha-\varkappa_2, \alpha-\varkappa_1)$.

Let $$A(z)=\frac{f'(z)}{f(z)}=h'_\alpha(z)= \int_{-\pi}^{\pi} \frac{\alpha e^{-it}\, d\psi(t)}{(1-ze^{-it})^{\alpha+1}}, \quad z\in\D
.$$
Note that, by the  construction $\psi\not\in \Lambda_\gamma$ for any  $\gamma >\varkappa_1$, also the measure $d\psi_1(t)=e^{-it} d\psi(t)$ belongs to the same Lipschitz class $\Lambda_
 {\varkappa_1}$. Thus,  by~\cite[Theorem 3]{Ch1},
$$M(r,A)=O\left( \Bl \frac{1}{1-r} \Br^{\alpha+1-\varkappa_1}\right), \quad r\to 1^-,$$
 and the exponent cannot be reduced.  Since $z A(z)=\alpha h_{\alpha+1}(z)-\alpha h_\alpha(z)$ for all $z\in\D$, we have
 \begin{equation}\label{e:mra_asymp}
  \alpha \, M(r, h_{\alpha+1})- \alpha \, M(r,h_\alpha) \leq   r\, M(r,A)  \leq \alpha \, M(r, h_{\alpha+1})+ \alpha \, M(r,h_\alpha)
 \end{equation}
for all $0\leq r<1$.
 Hence \begin{equation}\label{e:log_ord_sol}
         \lambda_{M,\deg}(A)=\lambda_{M,\deg}(h_{\alpha+1})=\alpha +1-\varkappa_2, \quad \sigma_{M,\deg}(A)=\alpha+1-\varkappa_1,
       \end{equation} 
 and $f$ is a solution of the equation $f'- Af=0$. In this case
\begin{equation*}
\lambda_M(f) \in \left( \frac{\lambda_{M,\deg}(A)}{1}-1, \, \frac{\sigma_{M,\deg}(A)}{1}-1 \right).
\end{equation*}

Now, let $0<\alpha<\varkappa_1<\varkappa_2<1$. Then (\cite[Theorem 6]{ChBer19} or \cite[Theorem 3]{Ch1}) 
$M(r, h_\alpha)=O(1)$, so $\sigma_M(f)=\lambda_M(f)=0$. On the other hand, \eqref{e:mra_asymp} still holds. Therefore,
\eqref{e:log_ord_sol} is valid. In this case the coefficient is of irregular growth, while all solutions are of regular growth.

\section*{Acknowledgement}

The authors thank Professor Yurii Lyubarskii for the idea to use approximation of subharmonic functions  
to construct the analytic function A in Theorem~\ref{e:sol_limits2}.

\end{document}